\documentclass[11pt]{amsart}
\usepackage[T1]{fontenc}
\usepackage[utf8]{inputenc}
\usepackage{amssymb,mathtools,mathrsfs,booktabs,array}
\usepackage[shortalphabetic,initials,nobysame]{amsrefs}
\usepackage[colorlinks=true,linkcolor=red,urlcolor=magenta,citecolor=black]{hyperref}
\hypersetup{
 pdftitle={Bethe Cycles, Frobenius Identities, and Quantum/Classical Duality},
 pdfauthor={Peter Koroteev},
 pdfsubject={Kahler difference equations, cyclic vertices, and Newton-Adams cancellation},
 pdfkeywords={Ruijsenaars, roots of unity, Frobenius, dynamical difference equations, complete flags, quantum K-theory}
}
\newtheorem{theorem}{Theorem}[section]
\newtheorem*{maintheorem}{Main Theorem}
\newtheorem{lemma}[theorem]{Lemma}
\newtheorem{proposition}[theorem]{Proposition}
\newtheorem{corollary}[theorem]{Corollary}
\theoremstyle{definition}

\theoremstyle{remark}
\newtheorem{remark}[theorem]{Remark}

\numberwithin{equation}{section}
\newcommand{\Z}{\mathbb Z}

\newcommand{\C}{\mathbb C}

\DeclareMathOperator{\Fl}{Fl}
\DeclareMathOperator{\GL}{GL}

\DeclareMathOperator{\diag}{diag}

\newcommand{\dd}{\mathrm d}

\newcommand{\id}{\operatorname{id}}

\renewcommand{\eprint}[1]{\href{https://arxiv.org/abs/#1}{arXiv:#1}}
\BibSpec{article}{%
    +{}  {\PrintAuthors}                {author}
    +{,} { \textit}                     {title}
    +{.} { }                            {part}
    +{:} { \textit}                     {subtitle}
    +{,} { \PrintContributions}         {contribution}
    +{.} { \PrintPartials}              {partial}
    +{,} { }                            {journal}
    +{}  { \textbf}                     {volume}
    +{}  { \PrintDatePV}                {date}
    +{,} { \issuetext}                  {number}
    +{,} { \eprintpages}                {pages}
    +{,} { }                            {status}
    +{,} { \eprint}                     {eprint}
    +{}  { \parenthesize}               {language}
    +{}  { \PrintTranslation}           {translation}
    +{;} { \PrintReprint}               {reprint}
    +{.} { }                            {note}
    +{.} {}                             {transition}
    +{}  {\SentenceSpace \PrintReviews} {review}
}
\title[Bethe Cycles Frobenius Identities]{Bethe Cycles, Frobenius Identities, and Quantum/Classical Duality}

\author{Peter Koroteev}
\address{Beijing Institute for Mathematical Sciences and Applications,
Beijing, China}
\email{peter.koroteev@gmail.com}
\date{\today}

\begin{document}
\begin{abstract}
We study Frobenius action on the trigonometric Ruijsenaars--Schneider
system associated with cotangent bundles to complete flag varieties.
The construction continues the program initiated by Etingof and
Varchenko and developed in quantum $K$-theory by the author and
Smirnov.  We explain how the monodromy eigenvectors recently
constructed by Tarasov and Varchenko fit the quantum/classical
duality between Bethe equations and tRS level sets.  At a root of unity, all partial shifts cancel in the Newton--Adams combinations of the tRS Hamiltonians, leaving precisely the operators obtained by Frobenius action.  This proves compatibility of the joint eigenvalue problem with Frobenius action in every rank and at every root order.
\end{abstract}
\maketitle
\enlargethispage{2pt}
\setcounter{tocdepth}{1}
\tableofcontents

\section{Introduction}
\label{sec:intro}

\subsection{Motivation and the Frobenius Program}

Recent work on the arithmetic properties of integrable connections has uncovered
relations between Frobenius action and classical spectral geometry.
Etingof and Varchenko \cites{EV24a,EV24b} initiated this study for
periodic pencils of flat connections and their $p$-curvature.
In quantum $K$-theory, the author and Smirnov \cite{KS26} showed
that full-cycle difference transport at roots of unity has the
spectrum of quantum multiplication at powered parameters.
The $p$-adic approximations of vertex functions studied by
Smirnov and Varchenko \cite{SV23} provide another instance of the
arithmetic structure of these enumerative functions.

In the present paper we shall examine Frobenius action from the
point of view of the trigonometric Ruijsenaars--Schneider (tRS)
system \cites{RS86,R87} associated with the cotangent bundle to the complete flag variety $T^*\Fl_n$.
The goal is to understand how
the recently constructed monodromy eigenvectors of
Tarasov and Varchenko \cite{TV26} enter the quantum/classical duality
between Bethe equations and tRS level sets.  Throughout the paper,
the difference operators act on K\"ahler parameters, while the
equivariant parameters $a_1,\ldots,a_n$ determine their spectrum.

Our construction is also motivated by the qKZ/tRS correspondence
of the author and Zeitlin \cite{KZ21}.  In that work a weighted
sum of components of a geometric qKZ solution produces a tRS
eigenfunction, which is identified with a vertex function in
Theorem~5.11.  This gives a geometric form of the
Matsuo--Cherednik correspondence \cites{M92,C94}.
It also has a natural interpretation through $3d$ mirror symmetry.
The variety $T^*\Fl_n$ is self-mirror: mirror symmetry exchanges
its K\"ahler and equivariant parameters, together with the
corresponding transformation of the cotangent parameter
\cite{D21} (Theorem~1 and Section~2.7).  On the integrable-system
side, the tRS model enjoys the corresponding self $3d$ mirror
symmetry, realized as bispectral self-duality between its
K\"ahler and equivariant difference eigenvalue problems
\cite{K23} (Section~2.4).
The projection identities of \cite{KZ21} (Proposition~5.5 and
Theorems~5.6,~5.10) act in equivariant variables.  Here we work
with the K\"ahler difference equations and use the vertex
eigenfunction theorem of \cite{K21} (Theorem~2.10) as our
geometric input.  The normalization is related to that of
\cite{D21} (Section~2.6, Propositions~1--3) below.
For a review of quantum/classical duality and the relation between
quantum geometry and integrable systems, see \cite{K23}.

In the qKZ/tRS construction of \cite{KZ21}, one starts with a
fundamental matrix of the quantum difference connection.  Its columns are
independent vector-valued solutions; its rows are coordinates in
a chosen basis, indexed for a complete flag by $\sigma\in S_n$.
Fix one column, labeled by $\tau$, and write its entries as
$I_{\sigma,\tau}$.  The projection of this vector solution gives
the scalar vertex.  After including the projection coefficients
in the components, this is $V_\tau=\sum_\sigma I_{\sigma,\tau}$.
Thus the component index and the choice of solution are different
indices.  We suppress $\tau$ when working with one column and
write $V=\sum_\sigma I_\sigma$.  Individual shifts mix the
entries through matrix transport, whereas the projected vertex
obeys the scalar tRS equations.  Our main theorem proves that
this eigenvalue problem is compatible with Frobenius action for
every complete flag and every root order.  The explicit descendant
periods used for the finite-sum presentation are defined in
Subsection~\ref{intro:components} below.

At a primitive root of unity, products of tRS Hamiltonians contain
many terms with partial shifts.  The Newton--Adams combinations
cancel these terms and leave only full turns around the coordinate
orbits.  We prove this operator identity in arbitrary rank and
then apply it to joint eigenfunctions.  For the finite-sum
construction, we start with the integral representation of the
vertex and replace integration in each auxiliary variable by a
weighted sum over $x,qx,\ldots,q^{p-1}x$.  Comparing the integrand
at two consecutive points gives the rational function which
propagates the weights.  Their periodicity is precisely the
Bethe condition at the powered parameters.
The resulting full-cycle multipliers lie on the corresponding
classical tRS level set and give a direct
finite-sum realization of the Frobenius equations.  For
$T^*\mathbb P^1$ both the ordinary equations and the component
connection can be calculated explicitly.

\subsection{K\"ahler Shifts and tRS Hamiltonians}

Let us first fix the geometric conventions and the normalization
of the difference operators.  We use the polarization convention
of \cite{KS26}, $TX=P+\hbar^{-1}P^*$, and fix a square root
$\hbar^{1/2}$.  Our parameters are
\[
\begin{array}{c|l}
q=\omega_p & \text{primitive root of unity of order }p\geq1\\
\hbar & \text{fiber dilation weight}\\
a_1,\ldots,a_n & \text{torus equivariant parameters}\\
z_1,\ldots,z_{n-1} & \text{K\"ahler parameters in the geometric connection}\\
x_{r,u} & \text{tautological roots and internal summation variables}.
\end{array}
\]
The order $p$ is not required to be prime.

Introduce homogeneous K\"ahler coordinates by
\begin{equation}\label{intro:kahler-coordinates}
 z_r=q\frac{\zeta_r}{\zeta_{r+1}},\qquad 1\leq r<n.
\end{equation}
Their simultaneous scale supplies the central character of
the $\GL(n)$ system.  Denote the shift $\zeta_i\mapsto q\zeta_i$
by $p_i$.  Then
\begin{equation}\label{intro:shifts}
 p_i\zeta_j=q^{\delta_{ij}}\zeta_jp_i,
 \qquad p_i z_r=q^{\delta_{ir}-\delta_{i,r+1}}z_rp_i,
 \qquad p_i a_j=a_jp_i.
\end{equation}
For a subset $I$, put $p_I=\prod_{i\in I}p_i$.  The
Macdonald--Cherednik tRS Hamiltonians in our normalization are
\begin{equation}\label{intro:hamiltonians}
 H_r=\sum_{|I|=r}b_Ip_I,\qquad
 b_I(\boldsymbol\zeta;\hbar)=\hbar^{-r(n-r)/2}
 \prod_{\substack{i\in I\\j\notin I}}
 \frac{\hbar\zeta_i-\zeta_j}{\zeta_i-\zeta_j}.
\end{equation}
In particular, $H_0=1$ and $H_n=p_1\cdots p_n$.

For generic $q$, the normalized fixed-point vertex $V_\tau$,
with its homogeneous character included, satisfies
\begin{equation}\label{intro:vertex-equations}
 H_rV_\tau=e_r(a_1,\ldots,a_n)V_\tau,
 \qquad H_nV_\tau=(a_1\cdots a_n)V_\tau.
\end{equation}
These equations were proved in \cite{K21} (Theorem~2.10).
The comparison with the modified vertex of \cite{D21} is given
in Section~\ref{flag:normalization}.  Notice that the homogeneous
character is needed to retain the eigenvalue $a_1\cdots a_n$
when the vertex is written in ratio coordinates.

\subsection{Orbit Graphs and Finite Sums}

Let us describe the finite graphs which appear in the calculation.
At $q^p=1$, multiplication of a nonzero coordinate by $q$ gives a
cycle of length $p$.  For the homogeneous K\"ahler coordinates,
the vertices are $q^{\boldsymbol m}\boldsymbol\zeta^0$ and the
directed edges are the shifts $p_i$.  The graph is a product of
$n$ cycles.  A term $b_Ip_I$ advances the coordinates indexed by
$I$ and assigns the coefficient $b_I$ to this step.  Products of
Hamiltonians therefore sum the weights of directed paths.
A full turn returns the coordinate to its starting value, but
may multiply a section by a nontrivial constant.  This constant
is recorded by $p_i^p$.

The finite vertex sum uses the same construction for its internal
variables.  Fix the K\"ahler parameters and choose nonzero initial
roots $x_\alpha^0$, with $R=n(n-1)/2$.  Independent shifts of
these roots give $p^R$ summation points.  A step $T_\alpha$
changes one root and carries the phase ratio $\Phi_\alpha$.
Since the ratios come from a common phase, they satisfy
\[
 \Phi_\alpha(\boldsymbol x)\Phi_\beta(T_\alpha\boldsymbol x)
 =\Phi_\beta(\boldsymbol x)\Phi_\alpha(T_\beta\boldsymbol x).
\]
Thus the weight of a path is unchanged when two adjacent steps
are interchanged.  To obtain periodic weights on the finite grid,
the product along every coordinate cycle must also equal one.
These conditions are the Bethe equations at the parameters
obtained by Frobenius action, $\mathcal B_{p,\alpha}=1$.
Here $\mathcal B_{p,\alpha}$ is the product of the $p$ successive
phase ratios in direction $\alpha$, defined in
\eqref{cyc:norms}; for one internal variable we write
$\mathcal B_p$, as in Section~\ref{sec:two}.

With these conditions imposed, a weighted total difference
cancels after summation: consecutive terms cancel along each
coordinate line, and the last term agrees with the first.
This elementary observation gives the finite functional and its
full-cycle multipliers.  The classical tRS/Bethe correspondence
then determines the eigenvalues after Frobenius action.  In rank one
we shall also use the same calculation to verify the ordinary
tRS equations and the first-order component system.

\subsection{From Component Solutions to Vertex Functions}
\label{intro:components}

Let us describe the integral meaning of the component presentation.
For a difference connection $p_i\Psi=\mathbf M_i\Psi$, a
fundamental solution is an invertible matrix $\Psi$ whose columns
are vector-valued solutions.  In an integral construction, a
common scalar phase is multiplied by a vector of rational weight
functions satisfying the connection identities and integrated
over a chosen cycle; the resulting period vector is a
vector-valued solution.  A collection of independent period vectors
provides the columns of a fundamental matrix.  Tarasov and Varchenko give the finite-sum
version of this construction in \cite{TV26} (Sections~3.2--3.6):
the weight-function identities ensure flatness, and nonzero closed
weighted sums give monodromy eigensections.  For the multiplicative
qKZ equations, see \cite{TV26} (Sections~5.4--5.5, in particular
Corollary~5.7 and equation~(5.19)).

In our notation, $\mathcal S_\Gamma[f]$ is the normalized
localization integral representing the $K$-theoretic quasimap
vertex of $T^*\Fl_n$, with descendant class $f$ inserted.
The integration variables are the auxiliary tautological roots
$x_{r,u}$; the scalar phase is \eqref{cyc:flag-phase}, and a
contour or residue prescription $\Gamma$ selects a vertex branch.
See \cite{K21} (equations~(2.13)--(2.16)) for the vertex integral.
The cycle, phase, homogeneous character, and normalization remain
fixed while the insertion varies.  We choose one such cycle and
suppress its label, writing $\mathcal S[f]$.

The explicit components of the finite presentation are descendant
periods for this common choice.  Set
$L_r=\prod_{u=1}^r x_{r,u}$, $L_0=1$, and
$\beta_r=L_r/L_{r-1}$.  For distinct $a_i$, define
\begin{equation}\label{intro:component-weights}
 w_\sigma(\boldsymbol x;\boldsymbol a)
 =\prod_{r=1}^{n-1}\prod_{s=r+1}^{n}
   \frac{\beta_r-a_{\sigma(s)}}{a_{\sigma(r)}-a_{\sigma(s)}},
 \qquad \sigma\in S_n.
\end{equation}
Successive Lagrange interpolation gives $\sum_\sigma w_\sigma=1$.
Applying the common vertex functional $\mathcal S$ gives
\begin{equation}\label{intro:component-sum}
 V=\mathcal S[1]
       =\sum_{\sigma\in S_n} I_\sigma,
       \qquad I_\sigma=\mathcal S[w_\sigma].
\end{equation}
Thus $\boldsymbol I_\Gamma=(\mathcal S_\Gamma[w_\sigma])_{\sigma\in S_n}$
is one vector of descendant periods.  The index $\sigma$ labels
its coordinates, while $\Gamma$ labels the chosen period vector.
In particular, $\mathcal S[1]$ is the vertex with unit insertion,
rather than the constant function one.  At a root of unity, we
use the weighted finite functional \eqref{cyc:functional-definition}
in place of the contour integral, following the discrete integration
construction of \cite{TV26} (Section~3.2).  The phase and its
normalization are described in Section~\ref{cyc:construction}.

This presentation allows us to apply the operator identities
directly to the component sum.  Once $V$ satisfies the joint
tRS equations, the same sum satisfies their Frobenius counterparts.
The rank-one calculation is given in Section~\ref{sec:two}.
For arbitrary flags, the descendant decomposition and the
full-cycle equations for each $I_\sigma$ are proved in
Section~\ref{cyc:explicit-components}.

\subsection{Bloch Solutions and Full-Cycle Transport}

A difference equation with periodic coefficients may have
solutions which acquire a multiplier after one period.
Such solutions are called Bloch solutions.  To see this in our
setting, write
$\boldsymbol\zeta(\boldsymbol m)=q^{\boldsymbol m}\boldsymbol\zeta^0$
and consider the system
\begin{equation}\label{intro:dynamical-system}
 \boldsymbol I(\boldsymbol m+e_i)
 =\mathbf M_i(\boldsymbol\zeta(\boldsymbol m))
       \boldsymbol I(\boldsymbol m),\qquad \boldsymbol m\in\Z^n.
\end{equation}
The product of $p$ consecutive transport matrices gives the
full-cycle monodromy.  On a common eigenline of these operators,
with nonzero eigenvalues $\lambda_i$, we have
\begin{equation}\label{intro:Bloch}
 \boldsymbol I(\boldsymbol m+pe_i)=\lambda_i\boldsymbol I(\boldsymbol m),
 \qquad V(\boldsymbol m+pe_i)=\lambda_iV(\boldsymbol m).
\end{equation}
For example, the scalar equation
$f(m+1)=b(q^mz_0)f(m)$ gives
\[
 f(m+p)=\left(\prod_{d=0}^{p-1}b(q^dz_0)\right)f(m).
\]
Taking $q=-1$ and $b(z)=\gamma(1-cz)$, we obtain
$\lambda=\gamma^2(1-c^2z_0^2)$ wherever this multiplier is nonzero.
The coordinate has returned to its starting point, while the
solution remembers its monodromy.  We retain this information
in the formal difference algebra by keeping $p_i^p$ as a shift
operator.  For the finite vertex sums, its multiplier will be
calculated directly from the phase ratios.

\subsection{Main Results}

The symmetric functions of powered spectral parameters can be
expressed as integral polynomials in the original elementary
symmetric functions.  Define these Newton--Adams polynomials by
\begin{equation}\label{intro:Adams}
 \mathcal A_{p,r}(e_1(y),\ldots,e_n(y))=e_r(y_1^p,\ldots,y_n^p).
\end{equation}
For $n=2$ and $r=1$, we recover the Newton polynomials
$P_p(E,D)$, where $P_p(u+v,uv)=u^p+v^p$.
Let $H_r^{[p]}$ denote the tRS Hamiltonian obtained by applying the Frobenius action to its coordinates, coupling, and shifts.  Our main
result identifies this operator with the Newton--Adams polynomial
in the original Hamiltonians.  It holds for the full commuting
family, for every complete flag, and for every root order.

\begin{maintheorem}
Let $n,p\geq1$ and let $q$ be a primitive $p$th root of unity.
The Hamiltonians \eqref{intro:hamiltonians} satisfy
$\mathcal A_{p,r}(H_1,\ldots,H_n)=H_r^{[p]}$, where
\begin{equation}\label{intro:main-identity}
 H_r^{[p]}
 =\hbar^{-pr(n-r)/2}\sum_{|I|=r}
   \prod_{\substack{i\in I\\j\notin I}}
   \frac{\hbar^p\zeta_i^p-\zeta_j^p}{\zeta_i^p-\zeta_j^p}\,p_I^p,
 \qquad 1\leq r\leq n.
\end{equation}
If $H_rV=e_r(a_1,\ldots,a_n)V$ for all $r$, then
\begin{equation}\label{intro:main-eigenproblem}
 H_r^{[p]}V=e_r(a_1^p,\ldots,a_n^p)V,
 \qquad 1\leq r\leq n.
\end{equation}
Thus the joint tRS eigenvalue equations are compatible with
Frobenius action for every complete flag and every root order.
\end{maintheorem}

The main achievement is the cancellation identity in the formal
difference algebra.  The classical tRS/Bethe correspondence already
determines the spectral level sets after applying the Frobenius action to
the parameters.  At the operator level, one must also account for
all partial shifts in products of Hamiltonians.  We prove that
their coefficients vanish in the Newton--Adams combinations and
calculate the surviving full-cycle coefficients exactly.
Theorem~\ref{flag:frobenius} establishes this in arbitrary rank using affine Hecke centrality
and bounds on the shift exponents.  For $T^*\mathbb P^1$, the
weighted matching calculation gives a direct proof for every $p$;
in arbitrary rank, Proposition~\ref{flag:direct-order-two} gives
an additional direct proof at $q=-1$.  These are identities in
characteristic zero, including even and composite root orders.

Applying the operator identity to a joint eigenfunction gives
Corollary~\ref{flag:eigenfunction-frobenius}.  The ordinary vertex
equations of \cite{K21} (Theorem~2.10) supply the geometric input,
and the Newton--Adams polynomial sends their eigenvalues to
$e_r(\boldsymbol a^p)$.  Thus the Frobenius compatibility of the
vertex eigenvalue problem follows for every complete flag.

There is also a direct quantum $K$-theoretic interpretation.
According to \cite{KPSZ21} (Theorem~4.6), the ratios
$L_d/L_{d-1}$ give the eigenvalues of quantum multiplication by
quantum determinant classes and their duals.  After the
normalization \eqref{cyc:external-ratios}, these are the tRS
momenta.  A geometric K\"ahler shift inserts $L_d$ into the vertex,
so its full-cycle multiplier is $L_d^p$, in agreement with
\cite{KS26}.  The finite functional is constructed by the cyclic
summation mechanism of Tarasov--Varchenko \cite{TV26}.

The cancellation used in the original dynamical/tRS correspondence
occurs in the auxiliary variables of the vertex representation.
K\"ahler shifts insert tautological classes into the phase, and
the identities underlying the scalar eigenvalue equation remove
internal twisted differences.  Integration over a compatible
cycle annihilates these differences.  In a finite orbit sum,
consecutive weighted terms cancel by reindexing, leaving a boundary
term proportional to $\mathcal B_{p,\alpha}-1$ in each internal direction.
Here $\mathcal B_{p,\alpha}$ is the product of the phase factors
around that coordinate cycle.  It is a rational function of the
equivariant parameters $a_i$, the K\"ahler parameters $z_r$, and
the tautological roots $x_{r,u}$, with $q$ and $\hbar$ fixed;
see \eqref{cyc:norms}.  The Bethe equations
$\mathcal B_{p,\alpha}=1$ close these cycles and remove
the boundary terms.  Their role is therefore to make this finite
summation consistent with the difference identities.  The
calculation for $T^*\mathbb P^1$ exhibits the component and tRS
residuals explicitly.  For complete flags,
Theorem~\ref{cyc:closure} proves that a closed weighted orbit
annihilates every internal twisted difference.

The cancellation specific to the Frobenius theorem takes place
in the external difference operators.  Expanding
$\mathcal A_{p,r}(H)$ produces terms with partial shifts
$p_i^{m_i}$, $0<m_i<p$.  At $q^p=1$ their total coefficients
vanish, leaving only $p_I^p$.  The remaining coefficients are
products of the original $b_I$ along full coordinate orbits;
the cyclotomic product turns them into the Frobenius substitutions
in \eqref{intro:main-identity}.  This identity holds before any
operator is applied to a vertex and uses no Bethe equations.
Its purpose is to identify the polynomial in quantum Hamiltonians
with the actual tRS family after Frobenius action.

For the finite flag sums, internal closure gives explicit Bloch
multipliers, and the known tRS/Bethe correspondence places them
on the classical level set at the powered parameters.  The
external operator identity then gives their Frobenius eigenvalue
equations in Theorem~\ref{cyc:finite-frobenius}, for every rank
and root order.  This direct construction and the eigenfunction
consequence of the main theorem give compatible realizations of
the same Frobenius spectral data.

\subsection{Structure of the Paper}

In Section~\ref{sec:two} we study the vertex for
$T^*\mathbb P^1$ and its tRS equations.  We begin with the
calculation at $q=-1$, derive the Bethe equation from the
boundary term, and construct the first-order component system.
We then perform the finite-sum calculation for arbitrary $p$.
The operator identity is proved using a weighted graph on the
ratio orbit $s_j=q^js_0$.  Matchings give the Newton corrections,
while the two full winding cycles give the Frobenius terms.

Section~\ref{sec:flag} is devoted to complete flags.  We express
products of Hamiltonians as weighted sequences of subsets and
prove the cancellation of partial shifts using affine Hecke
centrality.  At $q=-1$ we give a second proof based on a signed
bipartition identity.  We then discuss the classical spectral
fiber and construct finite vertex sums from the flag phase.
The Bethe closure conditions give their Frobenius eigenvalue
equations, and the explicit descendants give the decomposition
$V=\sum_\sigma I_\sigma$ in every rank.

\subsection{Acknowledgements}

I would like to thank Andrey Smirnov for useful discussions.
I acknowledge support from the Beijing Natural Science Foundation
(BJNSF) under grant IS26008.  I used ChatGPT-6.1 to study the details
of the cancellations after applying the Frobenius action.

\section{The tRS Equations for \texorpdfstring{$T^*\mathbb P^1$}{T*P1}}
\label{sec:two}

The simplest complete flag variety is $\mathbb P^1$, and its
cotangent bundle $T^*\mathbb P^1$ describes a two-particle tRS
model.  We shall work out this example in detail.
We keep the framing parameters $a_1,a_2$ fixed and use the global
parameter $\hbar$ in the polarization
$TX=P+\hbar^{-1}P^*$.  The difference operators below act on the
K\"ahler variables.  We shall use the notation
\[
 E=a_1+a_2,\qquad D=a_1a_2,\qquad
 (b;q)_d=\prod_{j=0}^{d-1}(1-bq^j),
\]
and write $p_z$ for the shift $z\mapsto qz$.  In particular,
$p_zz=qzp_z$.

\subsection{The Vertex and Hamiltonians}
\label{two:vertex}

Let us begin with the relation between the vertex and the
two-particle tRS model.  The generic-$q$ statement for normalized
vertices of complete flags is given in
\cite{K21} (Theorem~2.10).  In rank one we can derive the equation
from the series itself and describe the character which brings it
to the centered tRS normalization.

We choose the K\"ahler coordinate $z$ so that the shifted parameter
in the rank-one phase of \cite{KS26} (\S2.1) is $z/\hbar$.
At the fixed point $a_i$, where $\{i,j\}=\{1,2\}$, the bare
vertex with constant coefficient one reads
\begin{equation}\label{two:vertex-series}
 V_i^{\mathrm{bare}}(z)=\sum_{d\geq0}
 \frac{(\hbar;q)_d(\hbar a_i/a_j;q)_d}
      {(q;q)_d(qa_i/a_j;q)_d}
 \left(\frac z\hbar\right)^d.
\end{equation}
The framing parameters enter as coefficients.  Introduce a
difference character $e_{a_i}$ with $p_ze_{a_i}=a_i e_{a_i}$ and
write $\mathcal V_i=e_{a_i}V_i^{\mathrm{bare}}$.  Such a character
may be regarded either as a formal symbol with the prescribed
shift action or as a section of a rank-one difference local system.

It will be convenient to introduce
\begin{equation}\label{two:RJ}
 R(y)=(y-a_1)(y-a_2),\qquad
 J(y)=(\hbar y-a_1)(\hbar y-a_2).
\end{equation}
The scalar difference operator associated with the vertex is
\begin{equation}\label{two:native-operator}
 \mathcal L
 =R(p_z)-\frac z\hbar J(p_z)
 =(1-\hbar z)p_z^2-E(1-z)p_z
       +D(1-z/\hbar).
\end{equation}
Here and below coefficients precede the shifts.  The recurrence
of \eqref{two:vertex-series} is precisely the equation of this
operator, as follows.  Its relation to the tRS Hamiltonians will
be given by \eqref{two:operator-conversion}.

\begin{proposition}\label{two:vertex-equation}
For generic parameters, the formal series
\eqref{two:vertex-series} obeys $\mathcal L\mathcal V_i=0$.
\end{proposition}

\begin{proof}
Write $V_i^{\mathrm{bare}}=\sum_{d\geq0}v_dz^d$ and compare the
coefficients of $e_{a_i}z^d$.  For $d\geq1$ the equation becomes
\[
 R(a_iq^d)v_d=\hbar^{-1}J(a_iq^{d-1})v_{d-1}.
\]
Solving for the ratio of consecutive coefficients gives
\[
 \frac{v_d}{v_{d-1}}
 =\frac1\hbar
   \frac{(1-\hbar q^{d-1})(1-\hbar(a_i/a_j)q^{d-1})}
        {(1-q^d)(1-(a_i/a_j)q^d)},
\]
which agrees with \eqref{two:vertex-series}.  The constant term
of the equation vanishes since $R(a_i)=0$.
\end{proof}

The phase of the vertex provides another description of this
equation.  Its ratio under an internal shift is
\begin{equation}\label{two:phase-ratio}
 \phi(x;z)=\frac z\hbar
       \prod_{i=1}^2\frac{1-\hbar x/a_i}{1-qx/a_i}
       =\frac z\hbar\frac{J(x)}{R(qx)}.
\end{equation}
For a shift-invariant period functional we therefore have
$\langle R(x)\rangle=(z/\hbar)\langle J(x)\rangle$.
A K\"ahler shift inserts $x$, so this relation gives
\eqref{two:native-operator}.  At a root of unity the same argument
will be realized by a finite sum, whose boundary term we shall
compute explicitly.

For a primitive $p$th root of unity the vertex equation admits
the finite cyclic construction below.  The generic series needs
regularization: its coefficient of $z^p$ contains
$(1-q^p)^{-1}$.  The regularized fundamental matrix is constructed
in \cite{KS26} (Theorem~1.2).  Here we construct finite solutions
directly from the rational phase ratios.

\label{two:kahler-hamiltonian}

To pass to homogeneous K\"ahler coordinates, let
\begin{equation}\label{two:kahler-dictionary}
 s=\frac{\zeta_1}{\zeta_2}=\frac zq,
 \qquad
 p_i g(\zeta_1,\zeta_2)
 =g(\zeta_1,\ldots,q\zeta_i,\ldots,\zeta_2).
\end{equation}
The shifts $p_i$ leave $a_1,a_2$ and $\hbar$ unchanged.
Introduce the coefficients
\begin{equation}\label{two:AB}
 A_{\hbar}(s)=\frac{1-\hbar s}{\hbar^{1/2}(1-s)},
 \qquad
 B_{\hbar}(s)=\frac{\hbar-s}{\hbar^{1/2}(1-s)}.
\end{equation}
Then the centered Hamiltonians take the form
\begin{equation}\label{two:hamiltonians}
 H_1=A_{\hbar}(s)p_1+B_{\hbar}(s)p_2,
 \qquad H_2=p_1p_2.
\end{equation}
Let $\chi$ be a homogeneous character satisfying
\begin{equation}\label{two:homogeneous-character}
 p_1\chi=\hbar^{1/2}\,\chi,
 \qquad p_2\chi=D\hbar^{-1/2}\chi.
\end{equation}
By a character we mean a one-dimensional representation of the
lifted K\"ahler shift lattice $\Z^2$.  Above a base point
$\boldsymbol\zeta^0$ we have $\zeta_i=q^{m_i}\zeta_i^0$, and the
character is given by
\[
 \chi(m_1,m_2)=\chi(0,0)
       \hbar^{m_1/2}(D/\hbar^{1/2})^{m_2}.
\]
The constants in \eqref{two:homogeneous-character} are therefore
its eigenmultipliers under the two K\"ahler $q$-shifts.  For
generic $q$, a choice of logarithms represents $\chi$ locally by
a possibly multivalued monomial in $\zeta_1,\zeta_2$.  The formal
rank-one local system description applies as well.  At a root of
unity the coordinates return after $p$ steps, whereas the section
may acquire a Bloch multiplier.
For $V(\zeta)=\chi(\zeta)\mathcal V(qs)$ this gives the following
actions on $\mathcal V$:
$p_1=\hbar^{1/2}\,p_z$ and
$p_2=D\hbar^{-1/2}p_z^{-1}$.
Substitution into the Hamiltonian gives
\begin{align}\label{two:operator-conversion}
 \chi^{-1}(H_1-E)\chi
 &=\frac{1-\hbar z/q}{1-z/q}p_z
   +D\frac{1-z/(q\hbar)}{1-z/q}p_z^{-1}-E\notag\\
 &=p_z^{-1}(1-z)^{-1}\mathcal L.
\end{align}
The order of multiplication in the last expression is essential.
We conclude that
\begin{equation}\label{two:vertex-tRS}
 \mathcal L\mathcal V=0
 \quad\Longrightarrow\quad
 H_1V=EV,\qquad
 H_2V=DV.
\end{equation}
For the vertex at a fixed point, the character $\chi e_{a_i}$
has multipliers
$\hbar^{1/2} a_i$ and $a_j/\hbar^{1/2}$.
Thus the recurrence of the vertex series produces the tRS
equations in K\"ahler parameters.

\subsection{Bethe Equations from the Eigenvalue Problem}
\label{two:quadratic-cancellation}
\label{two:finite-two}

The first nontrivial root of unity is $q=-1$.  Here the vertex
is a sum of two terms, and its tRS equation leaves a boundary
condition which is precisely the Bethe equation.  The combination
$H_1^2-2H_2$ then gives the equation precomposed with the
Frobenius action.

Let $q=-1$, hence $s=-z$.  We keep $x\ne0$ fixed along the
K\"ahler orbit $z\mapsto-z$ and assume $R(x)R(-x)\ne0$.
The partial weights $\mu_0,\mu_1$ and the full-cycle factor
$\mathcal B_2$ are
\begin{align}\label{two:C-two}
 \mu_0&=1,\qquad
 \mu_1=\frac z\hbar\frac{J(x)}{R(-x)},\notag\\
 \mathcal B_2=\phi(x;z)\phi(-x;z)
   &=\left(\frac z\hbar\right)^2
          \frac{J(x)J(-x)}{R(x)R(-x)}.
\end{align}
With $e_x$ satisfying $p_ze_x=xe_x$ on the lifted orbit, the
native and normalized vertices read
\begin{equation}\label{two:finite-two-period}
 \mathcal V_x=e_x(1+\mu_1),\qquad
 V_x=\chi e_x(1+\mu_1).
\end{equation}
Since $\mu_1(-z)=-\mu_1(z)$, the K\"ahler shifts give
\begin{equation}\label{two:finite-two-shifts}
 p_1V_x=\chi e_x\hbar^{1/2}\,x(1-\mu_1),\qquad
 p_2V_x=\chi e_x\frac D{\hbar^{1/2}\,x}(1-\mu_1).
\end{equation}
Acting with the Hamiltonian, we obtain for $H_1V_x-EV_x$
\begin{align}\label{two:finite-two-direct-action}
 \frac{(H_1-E)V_x}{\chi e_x}
 &=\frac{[(1-\hbar s)x+D(1-s/\hbar)/x](1-\mu_1)}{1-s}
       -E(1+\mu_1)\notag\\
 &=\frac{R(x)+(z/\hbar)J(x)
          -\mu_1R(-x)-(z/\hbar)\mu_1J(-x)}{x(1-s)}\notag\\
 &=\frac{(1-\mathcal B_2)R(x)}{x(1-s)}.
\end{align}
Notice that the middle terms cancel by
$\mu_1R(-x)=(z/\hbar)J(x)$, while the last term is
$-\mathcal B_2R(x)$.  Homogeneity also gives $H_2V_x=DV_x$.
The adjacent terms of the finite sum have therefore canceled,
and the two endpoints agree when $\mathcal B_2=1$.
For the native scalar equation the same calculation gives
\begin{equation}\label{two:finite-two-defect}
 \mathcal L\mathcal V_x=e_x(1-\mathcal B_2)R(x).
\end{equation}

In terms of $U=x^2$, $A_i=a_i^2$, and $Z=z^2$, the closure
condition $\mathcal B_2=1$ becomes the Bethe equation
\begin{equation}\label{two:Bethe-two}
 (U-A_1)(U-A_2)
 =Z\hbar^2(U-\hbar^{-2}A_1)(U-\hbar^{-2}A_2).
\end{equation}
On each branch, $x$ remains fixed under $z\mapsto-z$, and
\eqref{two:finite-two-direct-action} gives $H_1V_x=EV_x$.
The polynomial parts at the two orbit points are $1+\mu_1$ and
$1-\mu_1$.  They cannot vanish simultaneously, so the resulting
cyclic section is nonzero.

Let us next consider the square of the Hamiltonian.  The relations
$p_1s=-sp_1$ and $p_2s=-sp_2$ give
\begin{align}\label{two:H-square-expanded}
 H_1^2
 &=A_{\hbar}(s)A_{\hbar}(-s)p_1^2
   +B_{\hbar}(s)B_{\hbar}(-s)p_2^2\notag\\
 &\quad+
 \bigl[A_{\hbar}(s)B_{\hbar}(-s)
       +B_{\hbar}(s)A_{\hbar}(-s)\bigr]p_1p_2.
\end{align}
The coefficient multiplying the mixed shift reduces to
\begin{equation}\label{two:mixed-two}
 \frac{(1-\hbar s)(\hbar+s)
              +(\hbar-s)(1+\hbar s)}
      {\hbar(1-s^2)}=2.
\end{equation}
For the two pure shifts we have
\[
 A_{\hbar}(s)A_{\hbar}(-s)=A_{\hbar^2}(s^2),
 \qquad
 B_{\hbar}(s)B_{\hbar}(-s)=B_{\hbar^2}(s^2).
\]
The mixed term is thus removed by subtracting $2H_2$:
\begin{equation}\label{two:newton-square}
 H_1^2-2H_2
 =A_{\hbar^2}(s^2)p_1^2+B_{\hbar^2}(s^2)p_2^2.
\end{equation}
On our finite vertex, the eigenmultipliers of $p_1^2$ and $p_2^2$
are $\hbar x^2$ and $D^2/(\hbar x^2)$, respectively.  It follows
that
\begin{equation}\label{two:square-direct-action}
 (H_1^2-2H_2)V_x
 =\frac{(1-\hbar^2Z)U+D^2(1-Z/\hbar^2)/U}{1-Z}\,V_x.
\end{equation}
By the Bethe equation \eqref{two:Bethe-two}, the coefficient on
the right is $a_1^2+a_2^2$, and therefore
\begin{equation}\label{two:square-on-wave}
 \bigl[A_{\hbar^2}(s^2)p_1^2
       +B_{\hbar^2}(s^2)p_2^2\bigr]V_x
 =(a_1^2+a_2^2)V_x.
\end{equation}
The closed finite vertex therefore solves both the original tRS
equations and their precomposition with the Frobenius action.
For the latter equations, the full-cycle multipliers and
\eqref{two:Bethe-two} already suffice; the ordinary equation
$H_1V_x=EV_x$ provides another derivation.

\subsection{Components of the K\"ahler Connection}
\label{two:component-connection}

The scalar vertex and its first K\"ahler shift determine a
first-order vector equation.  We shall choose an interpolation
basis in which the sum of the two components recovers the vertex.
This construction works for every $q$ and, after the homogeneous
gauge, relates the vector equation to the tRS system.  For the
finite cyclic vertex it gives the one-step component equations
as well as their full-cycle spectral relations.
Assume $a_1\ne a_2$ and introduce
\[
 w_1(y)=\frac{y-a_2}{a_1-a_2},\qquad
 w_2(y)=\frac{a_1-y}{a_1-a_2},\qquad w_1+w_2=1.
\]
For a solution $\mathcal V$ of the scalar vertex equation, define
\begin{equation}\label{two:interpolation-components}
 \Phi_1=\frac{p_z\mathcal V-a_2\mathcal V}{a_1-a_2},
 \qquad
 \Phi_2=\frac{a_1\mathcal V-p_z\mathcal V}{a_1-a_2}.
\end{equation}
Writing $\Phi=(\Phi_1,\Phi_2)^t$, we obtain the change of basis
\begin{equation}\label{two:interpolation-basis}
 \Phi=B
 \begin{pmatrix}\mathcal V\\p_z\mathcal V\end{pmatrix},
 \qquad
 B=\frac1{a_1-a_2}
 \begin{pmatrix}-a_2&1\\a_1&-1\end{pmatrix},
 \qquad
 B^{-1}=\begin{pmatrix}1&1\\a_1&a_2\end{pmatrix}.
\end{equation}
The scalar equation gives a K\"ahler connection on the column
$(\mathcal V,p_z\mathcal V)^t$ with companion matrix
\begin{equation}\label{two:companion}
 \mathsf M(z)=
 \begin{pmatrix}
 0&1\\
 -D(1-z/\hbar)/(1-\hbar z)&E(1-z)/(1-\hbar z)
 \end{pmatrix}.
\end{equation}
In the interpolation basis this becomes
\begin{equation}\label{two:component-system}
 p_z\Phi=K(z)\Phi,\qquad K(z)=B\mathsf M(z)B^{-1}.
\end{equation}
An explicit expression follows by setting
\[
 r_1=a_1(\hbar a_1-a_2),\qquad
 r_2=a_2(\hbar a_2-a_1).
\]
We find
\begin{equation}\label{two:component-matrix}
 K(z)=
 \begin{pmatrix}a_1&0\\0&a_2\end{pmatrix}
 +\frac{z(\hbar-1)}
        {\hbar(1-\hbar z)(a_1-a_2)}
   \begin{pmatrix}r_1&r_2\\-r_1&-r_2\end{pmatrix}.
\end{equation}
In this basis $K(0)=\diag(a_1,a_2)$, while the determinant is
$\det K(z)=D(1-z/\hbar)/(1-\hbar z)$.

For the finite period \eqref{two:cyclic-period} of order $p$,
the same change of basis gives
\begin{align}\label{two:finite-component-sum}
 \Phi_\beta
 &=e_x\sum_{j=0}^{p-1}\mu_j(x;z)w_\beta(q^jx),
       \qquad \beta=1,2,\notag\\
 V_x
 &=\chi e_x\sum_{j=0}^{p-1}\mu_j(x;z)
   =\chi(\Phi_1+\Phi_2)=I_1+I_2,\qquad I_\beta=\chi\Phi_\beta.
\end{align}
Indeed, the formulas for $\Phi_\beta$ follow from
\eqref{two:interpolation-components} and the shift relation
$p_z(e_x\mu_j)=e_xq^jx\mu_j$.  The scalar boundary term has the
following vector form:
\begin{equation}\label{two:component-defect}
 p_z\Phi-K(z)\Phi
 =\frac{e_x(1-\mathcal B_p)R(x)}
        {(a_1-a_2)(1-\hbar z)}
   \begin{pmatrix}1\\-1\end{pmatrix}.
\end{equation}
To see this, the residual of the companion equation is
$(0,\mathcal L\mathcal V_x/(1-\hbar z))^t$; multiplication by
$B$ gives the expression above.  The Bethe condition
$\mathcal B_p=1$ consequently yields the component system
\eqref{two:component-system}.

After applying the homogeneous gauge, the column $I=(I_1,I_2)^t$
obeys
\begin{equation}\label{two:homogeneous-component-system}
 p_1I=K_1(z)I,\qquad p_2I=K_2(z)I,\qquad
 K_1(z)=\hbar^{1/2}\,K(z),\qquad
 K_2(z)=\frac D{\hbar^{1/2}}K(z/q)^{-1}.
\end{equation}
We work where these matrices are defined and invertible.  The
argument $z/q$ in $K_2$ comes from solving
$\Phi(z)=K(z/q)\Phi(z/q)$ for the backward shift.
The two transports satisfy
\[
 K_2(qz)K_1(z)=K_1(z/q)K_2(z)=D\,\id.
\]
Hence $p_1p_2I=DI$.  The sum $V=I_1+I_2$ then satisfies the
scalar K\"ahler Ruijsenaars equations by
\eqref{two:operator-conversion}.

The mechanism can also be expressed directly in terms of the
rational weights.  Let $w=(w_1,w_2)^t$ and
$\varepsilon=(1,-1)^t$.  Their native difference identity reads
\begin{equation}\label{two:weight-difference}
 xw-K(z)w=\phi(x;z)U(qx;z)-U(x;z),\qquad
 U(x;z)=-\frac{R(x)}{(a_1-a_2)(1-\hbar z)}\varepsilon.
\end{equation}
The left-hand side equals
$[R(x)-(z/\hbar)J(x)]\varepsilon/
[(a_1-a_2)(1-\hbar z)]$, whereas
$\phi(x;z)R(qx)=(z/\hbar)J(x)$.
For the homogeneous equations the corresponding rational
functions are
\begin{equation}\label{two:homogeneous-differences}
 U_1=\hbar^{1/2}\,U,\qquad
 U_2=\frac{R(x)}
  {(a_1-a_2)\hbar^{1/2}\,x(1-z/(q\hbar))}
  \begin{pmatrix}a_2\\-a_1\end{pmatrix}.
\end{equation}
With $\rho_1=\hbar^{1/2}\,x$ and $\rho_2=D/(\hbar^{1/2}\,x)$,
we have
$\rho_iw-K_iw=\phi U_i(qx)-U_i(x)$.
Summing this difference over the finite orbit gives
\[
 (p_i-K_i)I=\chi e_x(\mathcal B_p-1)U_i(x;z).
\]
Thus both K\"ahler directions have the same Bethe boundary
obstruction.

It remains to examine the row $\ell=(1,1)$ which takes the
component sum.  In terms of $s=z/q$ we obtain
\[
 \ell K_1=\hbar^{1/2}\,(a_1,a_2),\qquad
 \ell K_2=
 \frac{E(1-s)\ell-(1-\hbar s)(a_1,a_2)}
      {\hbar^{1/2}(1-s/\hbar)}.
\]
The terms involving $(a_1,a_2)$ cancel in the combination
\begin{equation}\label{two:component-row-identity}
 A_\hbar(s)\ell K_1+B_\hbar(s)\ell K_2=E\ell.
\end{equation}
Together with the product relation for $K_1,K_2$, this identifies
the component sum with a solution of the full two-particle
hierarchy.  For our finite vertex, the remaining term is explicitly
\begin{equation}\label{two:tRS-boundary-defect}
 (H_1-E)V_x=
 \frac{\chi e_x(1-\mathcal B_p)R(x)}{x(1-s)},\qquad H_2V_x=DV_x.
\end{equation}
The Bethe equation therefore closes the component system and
the tRS eigenvalue problem simultaneously.

\subsection{Cyclic Elimination at Arbitrary Order}
\label{two:cyclic-determinant}

The Bethe equation can also be recovered by eliminating the
intermediate shifts of the vertex.  On a finite K\"ahler orbit,
this amounts to computing a determinant, whose vanishing gives
the Bethe equation at $(z^p,a_1^p,a_2^p,\hbar^p)$.

Let $q$ be a primitive $p$th root of unity, where $p\geq2$, and
write
\[
 Z=z^p,\qquad A_i=a_i^p.
\]
We use the $p$-dimensional cyclic representation of
$p_zz=qzp_z$ with $p_z^p=U$, in which multiplication by $z$ is
diagonal with entries $z,qz,\ldots,q^{p-1}z$.  Here $U$ is the
Bloch multiplier; the framing parameters remain fixed.
The matrix of $\mathcal L$ has a nonzero kernel precisely when
its determinant vanishes.  Thus the following calculation gives
the condition on the multiplier of a cyclic solution.

\begin{proposition}\label{two:native-determinant}
In the cyclic representation above, the scalar vertex operator
has determinant
\begin{equation}\label{two:powered-Bethe}
 \det\mathcal L
 =(U-A_1)(U-A_2)
   -Z\hbar^p(U-\hbar^{-p}A_1)(U-\hbar^{-p}A_2).
\end{equation}
\end{proposition}

\begin{proof}
Conjugating by $p_z$ sends $z$ to $qz$ and leaves the determinant
unchanged.  As a polynomial in $z$ it has degree at most $p$;
therefore only the constant term and the coefficient of $z^p$
can occur.  At $z=0$ we find
\[
 \det R(p_z)=(U-A_1)(U-A_2).
\]
The coefficient of $z^p$ is obtained by taking the determinant of
$-\hbar^{-1}\diag(1,q,\ldots,q^{p-1})J(p_z)$.
For primitive roots of odd or even order,
\[
 (-1)^p q^{p(p-1)/2}=-1.
\]
We also have
$\det J(p_z)=(\hbar^pU-A_1)(\hbar^pU-A_2)$,
so the leading term is
$-Z\hbar^{-p}(\hbar^pU-A_1)(\hbar^pU-A_2)$,
as claimed.
\end{proof}

The condition $\det\mathcal L=0$ is consequently the powered
Bethe equation for a nonzero cyclic solution.  We shall now
construct that solution on the regular locus of the phase ratios.

\subsection{The Finite-Sum Cancellation for Arbitrary \texorpdfstring{$p$}{p}}
\label{two:finite-general}

The finite period realizes the determinant condition through
cancellation of adjacent weighted terms.  Its boundary is measured
by $1-\mathcal B_p$, so closing the cycle gives a vertex solution.

Keep $x\ne0$ fixed along a K\"ahler orbit and assume
$x^p\ne a_i^p$.  For each integer $j\geq0$ let
\begin{equation}\label{two:cyclic-coefficients}
 \mu_j(x;z)
 =\prod_{r=0}^{j-1}\phi(q^rx;z)
 =\left(\frac z\hbar\right)^j
    \prod_{i=1}^2
      \frac{(\hbar x/a_i;q)_j}{(qx/a_i;q)_j},
 \qquad \mu_0=1.
\end{equation}
Here $\mu_j$ is the product of the phase ratios along the first
$j$ steps of $x,qx,\ldots,q^{p-1}x$.  In particular, the
transport around the entire cycle is
\begin{equation}\label{two:Bethe-holonomy}
 \mathcal B_p(x;z):=\mu_p(x;z)
   =\prod_{r=0}^{p-1}\phi(q^rx;z).
\end{equation}
A descendant $f$ of the internal variable is inserted into the
weighted sum over this orbit as follows:
\begin{equation}\label{two:finite-descendant-period}
 \langle f\rangle_{x,z}
   :=e_x\sum_{j=0}^{p-1}\mu_j(x;z)f(q^jx),
 \qquad p_ze_x=xe_x,
\end{equation}
where $f$ is regular at the orbit points.  The unit insertion
$1$ gives the finite unit-descendant period
\begin{equation}\label{two:cyclic-period}
 S_x(z)=\sum_{j=0}^{p-1}\mu_j(x;z),\qquad
 \mathcal V_x=e_xS_x=\langle1\rangle_{x,z}.
\end{equation}
The unit period therefore has the same phase weights as every
descendant, with no further factor in the internal variable.
Its cycle closes when $\mathcal B_p=1$; otherwise the sum has
the boundary term calculated below.  The interpolation components
are precisely $\Phi_\beta=\langle w_\beta\rangle_{x,z}$, and the
normalized vertex is $V_x=\chi\mathcal V_x$.
Using $\mu_j(x;qz)=q^j\mu_j(x;z)$, we obtain the moments
\begin{equation}\label{two:shift-moments}
 p_z^m\mathcal V_x
 =e_x\sum_{j=0}^{p-1}\mu_j(x;z)(q^jx)^m,
 \qquad m\geq0.
\end{equation}

These moments allow us to apply the scalar vertex operator to
\eqref{two:cyclic-period}.  The result is the arbitrary-order
version of \eqref{two:finite-two-defect}, relating the boundary
of the finite period to the determinant condition
\eqref{two:powered-Bethe}.

\begin{proposition}\label{two:boundary-defect}
The finite cyclic period satisfies
\begin{equation}\label{two:exact-defect}
 \mathcal L\mathcal V_x
 =e_x(1-\mathcal B_p)(x-a_1)(x-a_2).
\end{equation}
On the stated regular locus, it is therefore a solution of the
vertex equation if and only if $\mathcal B_p=1$.
\end{proposition}

\begin{proof}
The consecutive weights $\mu_j$ satisfy
\[
 \mu_{j+1}R(q^{j+1}x)=\frac z\hbar \mu_jJ(q^jx).
\]
Substituting \eqref{two:shift-moments} into the operator, we see
that consecutive terms cancel and only the endpoints remain:
\begin{align*}
 e_x^{-1}\mathcal L\mathcal V_x
 &=\sum_{j=0}^{p-1}\mu_jR(q^jx)
    -\sum_{j=0}^{p-1}\mu_{j+1}R(q^{j+1}x)\\
 &=R(x)-\mathcal B_pR(q^px)=(1-\mathcal B_p)R(x).
\end{align*}
Finally, $x^p\ne a_i^p$ ensures $R(x)\ne0$, which gives the
stated equivalence.
\end{proof}

We can perform the calculation directly with the tRS Hamiltonian.
On each summand the two shifts act by
$p_1(\chi e_x\mu_j)=\hbar^{1/2}\,q^jx\,\chi e_x\mu_j$ and
$p_2(\chi e_x\mu_j)=D(\hbar^{1/2}\,q^jx)^{-1}\chi e_x\mu_j$.
The Hamiltonian consequently acts as
\begin{align}\label{two:direct-tRS-action}
 \frac{(H_1-E)V_x}{\chi e_x}
 &=\frac1{1-s}\sum_{j=0}^{p-1}\mu_j
   \left[(1-\hbar s)q^jx+
     \frac{D(1-s/\hbar)}{q^jx}-E(1-s)\right]\notag\\
 &=\frac1{x(1-s)}\sum_{j=0}^{p-1}
   \left[q^{-j}\mu_jR(q^jx)
         -q^{-(j+1)}\mu_{j+1}R(q^{j+1}x)\right]\notag\\
 &=\frac{(1-\mathcal B_p)R(x)}{x(1-s)}.
\end{align}
The second equality follows by expanding $R,J$ and applying
$\mu_{j+1}R(q^{j+1}x)=(z/\hbar)\mu_jJ(q^jx)$ together with
$s=z/q$.  After telescoping, $q^{-p}=1$ makes the last endpoint
equal to $\mathcal B_pR(x)$.  The other Hamiltonian satisfies
$H_2V_x=DV_x$ on each term.  We have thus obtained the tRS
eigenvalue equation from the closed weighted cycle,
$\mathcal B_p=1$.

The cyclotomic product of the factors in
\eqref{two:cyclic-coefficients} gives
\begin{equation}\label{two:closure-product}
 \mathcal B_p=\frac{Z}{\hbar^p}
      \prod_{i=1}^2\frac{1-\hbar^px^p/A_i}{1-x^p/A_i}.
\end{equation}
Here $A_i=a_i^p$, $i=1,2$, are the Frobenius powers of the fixed
equivariant parameters, and $Z=z^p$.  These constants are distinct
from the coefficient function $A_\hbar(s)$ and its orbit values
$A_j$ introduced below.
Upon setting $U=x^p$, the condition $\mathcal B_p=1$ is exactly
the vanishing of the determinant in \eqref{two:powered-Bethe}.
Thus elimination of the intermediate K\"ahler shifts and closure
of the internal cycle give the same equation.

It remains to check that the cyclic section is nonzero.  On a
free orbit with $z\ne0$, its polynomial parts read
\[
 S_x(q^mz)=\sum_{j=0}^{p-1}q^{mj}\mu_j(x;z),
 \qquad 0\leq m<p.
\]
The inverse discrete Fourier transform recovers $\mu_0=1$.
Hence at least one coordinate of the orbit vector is nonzero,
although individual coordinates may vanish.  For every order
$p$ we therefore obtain a nonzero cyclic solution of the ordinary
tRS equations on the regular Bethe locus where their coefficients
are defined.  Its one-step component system is given by
\eqref{two:component-defect}--\eqref{two:homogeneous-component-system}
wherever the matrices are defined and invertible.

\subsection{The Operator Cancellation for Arbitrary \texorpdfstring{$p$}{p}}
\label{two:eigenfunction-Frobenius}

The cancellation in the square of the two-particle Hamiltonian
extends to a universal Newton identity at every root of unity.
We shall derive it from the orbit graph, independently of the
vertex or the Bethe equations.  For our finite vertex the resulting
Frobenius equation follows both from the ordinary eigenvalue
problem and from the full-cycle multipliers together with the
Bethe equation at powered parameters.

Recall that on a regular branch of $\mathcal B_p=1$ the normalized
finite vertex is
\begin{equation}\label{two:normalized-cyclic-vertex}
 V_x=\chi\mathcal V_x
     =\chi e_x\sum_{j=0}^{p-1}\mu_j(x;z)
     =\chi(\Phi_1+\Phi_2)=I_1+I_2.
\end{equation}
Its components obey \eqref{two:homogeneous-component-system}.
By \eqref{two:operator-conversion}, their sum satisfies
\[
 H_1V_x=(a_1+a_2)V_x,
 \qquad H_2V_x=a_1a_2V_x.
\]
The full-cycle K\"ahler shifts act as
\begin{equation}\label{two:Bloch-multipliers}
 p_1^pV_x=\hbar^{p/2}UV_x,
 \qquad
 p_2^pV_x=D^p\hbar^{-p/2}U^{-1}V_x.
\end{equation}
These equations hold for lifted sections, or equivalently on
Bloch fibres.  Generic multipliers do not give single-valued
functions on the finite coordinate orbit.

Introduce the Newton polynomials in two variables by
\begin{equation}\label{two:Newton}
 P_0(X,Y)=2,\quad P_1(X,Y)=X,\quad
 P_m(X,Y)=XP_{m-1}(X,Y)-YP_{m-2}(X,Y).
\end{equation}
Their defining property is
$P_m(b_1+b_2,b_1b_2)=b_1^m+b_2^m$.
At a root of unity the K\"ahler Hamiltonians satisfy the identity
\begin{equation}\label{two:Newton-identity}
 P_p(H_1,H_2)
 =A_{\hbar^p}(s^p)p_1^p+B_{\hbar^p}(s^p)p_2^p.
\end{equation}
We prove the universal operator identity \eqref{two:Newton-identity}
by a calculation on the $q$-orbit graph, independently of the chosen
vertex.  For $p=1$ the assertion is immediate, so first take $p\geq3$.
Label the vertices by $j\in\Z/p\Z$, corresponding to $q^js$,
and write $A_j=A_\hbar(q^js)$ and $B_j=B_\hbar(q^js)$.
The weighted directed edges are
\[
 j\xrightarrow{\ \alpha A_j\ }j+1,\qquad
 j\xrightarrow{\ \beta B_j\ }j-1,
\]
where $\alpha,\beta$ are independent Bloch parameters.  The
orientation of an entry goes from its row vertex to its column
vertex.  Multiplication by $f(s)$ is represented by
$\diag(f(q^js))$, and the shift representation is
\begin{equation}\label{two:orbit-graph}
 \begin{aligned}
 p_1&=\alpha T,\qquad p_2=\beta T^{-1},
       \qquad (Tv)_j=v_{j+1},\\
 \mathsf H(s)&=\alpha\diag(A_j)T+\beta\diag(B_j)T^{-1}.
 \end{aligned}
\end{equation}
One checks the shift relations
$p_1s=qs p_1$ and $p_2s=q^{-1}s p_2$, while
$H_2=\alpha\beta$.

Matrix powers of $\mathsf H$ enumerate directed walks, including
walks which revisit a vertex.  In the determinant
$\det(X-\mathsf H)$, the relevant objects are vertex-disjoint
permutation cycles.  On our nearest-neighbor graph these consist
of fixed vertices, two-cycles on edges, and the two oriented
cycles winding around the graph.

Disjoint two-cycles form a matching.  We index the undirected
edge $\{j,j+1\}$ by $j\in\Z/p\Z$, including the edge
$\{p-1,0\}$.  Its two-cycle contains the steps $j\to j+1$ and
$j+1\to j$.  Dividing their weight by $\alpha\beta$ gives
\begin{equation}\label{two:matching-edge-weight}
 w_j=A_jB_{j+1}
 =\frac{(1-\hbar q^js)(\hbar-q^{j+1}s)}
        {\hbar(1-q^js)(1-q^{j+1}s)}.
\end{equation}
For $0\leq k\leq\lfloor p/2\rfloor$, let
\begin{equation}\label{two:matching-index-set}
 \mathcal M_{p,k}
 =\bigl\{M\subset\Z/p\Z:
          |M|=k,\ M\cap(M+1)=\varnothing\bigr\}.
\end{equation}
Here $M+1=\{j+1:j\in M\}$ and indices are taken modulo $p$.
The element $M$ specifies $k$ edges $\{j,j+1\}$, one for each
$j\in M$.  The condition in the definition excludes a common
vertex between any two of them.  Thus $\mathcal M_{p,k}$
parametrizes the $k$-edge matchings of the cycle, each exactly once.

The permutation with these $k$ two-cycles and $p-2k$ fixed
vertices contributes
$(-\alpha\beta)^kX^{p-2k}\prod_{j\in M}w_j$.
Consequently, its total coefficient is determined by summing
$\prod_{j\in M}w_j$ over $M\in\mathcal M_{p,k}$.
The following lemma identifies this weighted sum with the number
of matchings.  These numbers will become the Newton coefficients
in \eqref{two:Newton-explicit}.

\begin{lemma}\label{two:matching-cancellation}
Let $p\geq3$, and let $q$ be a primitive $p$th root of unity.
For the weights \eqref{two:matching-edge-weight} and matchings
\eqref{two:matching-index-set}, we have
\begin{equation}\label{two:matching-sum}
 S_{p,k}(s):=\sum_{M\in\mathcal M_{p,k}}\prod_{j\in M}w_j
 =\frac{p}{p-k}\binom{p-k}{k},
 \qquad 0\leq k\leq\lfloor p/2\rfloor.
\end{equation}
For $k=0$ the only matching is $M=\varnothing$, with empty
product $1$ as its weight.
\end{lemma}

\begin{proof}
Since a matching uses a vertex at most once, the denominator of
its weight contains each factor $1-q^js$ at most once.  Hence
\begin{equation}\label{two:matching-polynomial}
 \begin{aligned}
 N_{p,k}(s)&:=(1-s^p)S_{p,k}(s)\\
 &=\sum_{M\in\mathcal M_{p,k}}
   \prod_{j\in M}
     \frac{(1-\hbar q^js)(\hbar-q^{j+1}s)}{\hbar}
   \prod_{v\notin M\cup(M+1)}(1-q^vs)
 \end{aligned}
\end{equation}
is a polynomial of degree at most $p$.  Rotation of the cycle
permutes its matchings, giving $N_{p,k}(qs)=N_{p,k}(s)$.
Since $q$ is primitive, only the constant term and the degree-$p$
coefficient can survive.

The constant coefficient of each summand in
\eqref{two:matching-polynomial} is one, and its degree-$p$
coefficient equals
\[
 (-1)^{p-2k}q^{0+1+\cdots+(p-1)}=-1,
\]
by $q^{p(p-1)/2}=(-1)^{p-1}$.  Thus $N_{p,k}(s)$ equals the
number of matchings multiplied by $1-s^p$.  Their number is
\[
 \binom{p-k}{k}+\binom{p-k-1}{k-1}
 =\frac p{p-k}\binom{p-k}{k}.
\]
The first term counts matchings without the edge $\{p-1,0\}$,
and the second counts those containing it; the latter is zero
for $k=0$.  This gives \eqref{two:matching-sum}.
When $p$ is even and $k=p/2$, the sum contains both alternating
perfect matchings, with total weight $2$.
\end{proof}

We can now evaluate the determinant of \eqref{two:orbit-graph}
using Lemma~\ref{two:matching-cancellation}.  A fixed vertex
contributes $X$, a two-cycle contributes $-\alpha\beta w_j$,
and a full winding cycle contributes the negative of its edge
product.  For the last assertion, the permutation sign
$(-1)^{p-1}$ combines with the $p$ off-diagonal factors to give
an additional $(-1)^p$, so the overall sign is negative at every
order.  We obtain
\begin{align}\label{two:cycle-cover-determinant}
 \det(X-\mathsf H(s))
 &=\sum_{k=0}^{\lfloor p/2\rfloor}
       (-1)^k S_{p,k}(s)(\alpha\beta)^kX^{p-2k}\notag\\
 &\hspace{1cm}-\alpha^p\prod_{j=0}^{p-1}A_j
                -\beta^p\prod_{j=0}^{p-1}B_j.
\end{align}
For even $p$, the perfect matchings occur in the first sum;
they are separate from the two winding cycles.
The matching numbers in \eqref{two:matching-sum} are the
coefficients of the Newton polynomial
\begin{equation}\label{two:Newton-explicit}
 P_p(X,Y)=\sum_{k=0}^{\lfloor p/2\rfloor}
   (-1)^k\frac p{p-k}\binom{p-k}{k}Y^kX^{p-2k}.
\end{equation}
The same expression follows from the recurrence
\eqref{two:Newton}.  We therefore identify the matching part of
the determinant with the Newton corrections.  Cayley--Hamilton
will translate this determinant calculation into cancellation of
mixed walks in the operator polynomial.  No decomposition of an
individual walk into disjoint dimers is involved.

The weights of the two winding cycles are the cyclotomic products
\begin{align}\label{two:winding-products}
 \prod_{j=0}^{p-1}A_j
 &=\frac{1-\hbar^ps^p}{\hbar^{p/2}(1-s^p)}
   =A_{\hbar^p}(s^p),\notag\\
 \prod_{j=0}^{p-1}B_j
 &=\frac{\hbar^p-s^p}{\hbar^{p/2}(1-s^p)}
   =B_{\hbar^p}(s^p).
\end{align}
Combining \eqref{two:matching-sum}--\eqref{two:winding-products}
gives the characteristic polynomial
\begin{equation}\label{two:characteristic-polynomial}
 \det(X-\mathsf H(s))
 =P_p(X,\alpha\beta)
   -A_{\hbar^p}(s^p)\alpha^p
   -B_{\hbar^p}(s^p)\beta^p.
\end{equation}

For $p=2$ both directed steps connect the same two vertices.
Their weights therefore add in the matrix entries:
\[
 \mathsf H(s)=
 \begin{pmatrix}
  0&\alpha A_0+\beta B_0\\
  \alpha A_1+\beta B_1&0
 \end{pmatrix}.
\]
The characteristic polynomial in this case reads
\[
 X^2-\alpha^2A_0A_1-\beta^2B_0B_1
       -\alpha\beta(A_0B_1+B_0A_1).
\]
Here \eqref{two:mixed-two} makes the mixed coefficient $2$,
while \eqref{two:winding-products} evaluates the pure products.
This proves \eqref{two:characteristic-polynomial} for $p=2$ as
well, with $P_2(X,Y)=X^2-2Y$.

It remains to recover the operator relation from this polynomial.
For fixed primitive root $q$, the matrix $\mathsf H(s)$ has size
$p\times p$ and entries in the commutative field
$\C(s,\alpha,\beta,\hbar^{1/2})$.  Specializing the parameters
away from its poles gives an ordinary complex matrix, whose
characteristic polynomial is
\[
 \chi_{\mathsf H}(X)=\det(X\,\id-\mathsf H(s))
   =X^p+c_1X^{p-1}+\cdots+c_p.
\]
The coefficients $c_j$ are
scalars in the coefficient field.  By Cayley--Hamilton, the matrix
satisfies its own polynomial, namely
\[
 \chi_{\mathsf H}(\mathsf H(s))
 =\mathsf H(s)^p+c_1\mathsf H(s)^{p-1}
       +\cdots+c_p\,\id=0.
\]
In this substitution, each power of $X$ in the expanded scalar
polynomial becomes a matrix power, while the constant coefficient
multiplies $\id$.  Thus the substitution takes place after
expanding the determinant, rather than within its entries.
One can see the relation directly by comparing coefficients of
$X$ in the adjugate identity
\[
 (X\,\id-\mathsf H(s))
       \operatorname{adj}(X\,\id-\mathsf H(s))
   =\chi_{\mathsf H}(X)\,\id.
\]
Beginning with the highest power, the coefficient equations
express the adjugate coefficients successively as polynomials in
$\mathsf H(s)$.  The constant-coefficient equation is then
exactly the Cayley--Hamilton relation above.  Using the explicit
polynomial \eqref{two:characteristic-polynomial}, it becomes
\begin{equation}\label{two:Newton-matrix-identity}
 P_p(\mathsf H(s),\alpha\beta\,\id)
 =\bigl[A_{\hbar^p}(s^p)\alpha^p
       +B_{\hbar^p}(s^p)\beta^p\bigr]\id.
\end{equation}
The matrix representing $H_2$ is $\alpha\beta\,\id$, so this
is \eqref{two:Newton-identity} in the representation
\eqref{two:orbit-graph}.  The independence of $\alpha,\beta$
is important: it keeps the full-cycle shifts
$p_1^p=\alpha^p$ and $p_2^p=\beta^p$ as formal multipliers,
instead of specializing them to one.  Indeed, the image of a
normal-ordered term $f(s)p_1^ap_2^b$ is
\[
 \alpha^a\beta^b\diag(f(q^js))T^{a-b}.
\]
Different pairs $(a,b)$ give different Laurent monomials in
$\alpha,\beta$.  The representation is therefore faithful on
finite Laurent sums of shifts with rational coefficients in $s$,
and the matrix relation proves the formal operator identity.
The matching terms give its Newton corrections, while the two
winding cycles give the Frobenius terms.

We can now apply \eqref{two:Newton-identity} to $V_x$ and obtain
\begin{equation}\label{two:Frobenius-on-wave}
 \bigl[A_{\hbar^p}(s^p)p_1^p
       +B_{\hbar^p}(s^p)p_2^p\bigr]V_x
 =(a_1^p+a_2^p)V_x.
\end{equation}
The full-cycle multipliers evaluate the two shifts explicitly:
\begin{equation}\label{two:powered-direct-action}
 \bigl[A_{\hbar^p}(Z)p_1^p+B_{\hbar^p}(Z)p_2^p\bigr]V_x
 =\frac{(1-\hbar^pZ)U+D^p(1-Z/\hbar^p)/U}{1-Z}\,V_x.
\end{equation}
Here $s^p=z^p=Z$, with $Z=1$ excluded from the regular locus.
The coefficient is $a_1^p+a_2^p$ by the powered Bethe equation.
To check this directly, subtract $(a_1^p+a_2^p)V_x$ from
\eqref{two:powered-direct-action}.  The difference is
\begin{equation}\label{two:powered-Bethe-residual}
 \frac{(U-A_1)(U-A_2)
       -Z\hbar^p(U-\hbar^{-p}A_1)(U-\hbar^{-p}A_2)}
      {U(1-Z)}\,V_x,
\end{equation}
and vanishes by \eqref{two:powered-Bethe} on $\mathcal B_p=1$.
For the second Hamiltonian, \eqref{two:Bloch-multipliers} gives
$p_1^pp_2^pV_x=D^pV_x$.  We have thus verified the full
two-particle eigenvalue problem precomposed with the Frobenius action
from the cycle multipliers and the Bethe equation alone.  The ordinary
equations $H_rV_x=e_r(a_1,a_2)V_x$ and the matrices $K_i$ are
not needed for this last calculation.

The ordinary eigenvalue equations give the same Frobenius
equations, as in the general statement of
Corollary~\ref{flag:eigenfunction-frobenius}.  The calculation
from cycle multipliers and powered Bethe roots extends to complete
flags in Theorem~\ref{cyc:finite-frobenius}.  In the two-particle
case we have, in addition, established nonvanishing of the finite
vertex and constructed its first-order component system
\eqref{two:homogeneous-component-system}.

\section{Complete Flags: Cycles and Newton--Adams Cancellation}
\label{sec:flag}

Let us now consider the $n$-body tRS model associated with complete
flags, generalizing the calculation of Section~\ref{sec:two}.
The shifts $p_i$ act on homogeneous K\"ahler
coordinates by $\zeta_i\mapsto q\zeta_i$.  When $q$ is a primitive
$p$th root of unity, their coordinate orbits are finite, although a
section on the lifted orbit may acquire a Bloch multiplier after a
full cycle.  In higher rank, the matching model is replaced by
sequences of subsets: each of the $p$ steps shifts the coordinates
in a chosen subset.  We shall show that the Newton--Adams combination
of their weights cancels every endpoint at which a coordinate has
been shifted only part of the way around its orbit.

The proof uses centrality in the Hecke algebra and the bounds on shift
exponents.  For each surviving full-cycle shift there is only one
contributing subset sequence, which determines its coefficient.
This gives the operator identity in every rank and at every root
order, and hence its eigenfunction consequence in
Corollary~\ref{flag:eigenfunction-frobenius}.  At order two we shall
also give a direct proof by pairing the collision residues of a
signed bipartition sum.

For the finite vertex sums we also need the internal grid of auxiliary
roots $x_{r,s}$, $1\leq s\leq r<n$, with $R=n(n-1)/2$ variables.
Starting from nonzero values $x_{r,s}^0$, replace each independently
by $q^{m_{r,s}}x_{r,s}^0$, where $m_{r,s}\in\Z/p\Z$.
The resulting grid has $p^R$ points; the K\"ahler parameters,
equivariant parameters, and $\hbar$ are held fixed.  Each edge carries
the phase ratio for the corresponding auxiliary shift.  Flatness
identifies the two products around an elementary square, while the
product around a full coordinate loop is its Bethe norm.  When these
norms equal one, the weights are periodic and total differences
vanish under summation.

The external operators act on the K\"ahler parameters of these sums,
leaving the internal roots fixed.  The tRS/Bethe correspondence
identifies their full-cycle multipliers with classical tRS momenta
at the powered parameters.  Combining this fact with the operator
identity gives Theorem~\ref{cyc:finite-frobenius}.  Thus Bethe closure
provides finite sums to which the operator identity applies, whereas
the identity itself and its consequence for joint eigenfunctions
hold independently of closure.

The section is organized as follows.
Subsections~\ref{flag:normalization}--\ref{flag:histories} introduce
the shift algebra, Adams polynomials, and subset sequences.
Hecke centrality in Subsection~\ref{flag:centrality-subsection}
then gives the operator identity and its eigenfunction consequence
in Subsection~\ref{flag:operator-identity}.
Subsection~\ref{flag:examples} supplies the direct order-two proof,
and Subsection~\ref{flag:Bloch} identifies the classical spectral
fiber through Bloch multipliers.  Starting with
Subsection~\ref{cyc:construction}, we turn to the internal flag phase
and its K\"ahler transport, impose the Bethe closure conditions,
and construct the finite sums satisfying the Frobenius equations
of Theorem~\ref{cyc:finite-frobenius}.  Finally,
Subsection~\ref{cyc:explicit-components} expresses the unit-insertion
sum through explicit descendant components.

\subsection{Parameters and the Formal Shift Algebra}
\label{flag:normalization}

We begin with the normalization of the Hamiltonians.  Let
$q=\omega_p$ be a primitive root of unity of order $p\geq1$;
the order need not be prime.  The tRS positions are the homogeneous
K\"ahler coordinates $\zeta_1,\ldots,\zeta_n$, while the equivariant
parameters $a_1,\ldots,a_n$ are fixed spectral data.  We keep the
cotangent parameter $\hbar\ne0$ fixed as well.  The auxiliary roots
$x_{r,s}$, $1\leq s\leq r<n$, enter the flag integrals but not the
operator identity.

The geometric K\"ahler parameters and their shifts are
\begin{equation}\label{flag:kahler-coordinates}
 z_r=q\,\frac{\zeta_r}{\zeta_{r+1}},\qquad
 p_i\zeta_j=q^{\delta_{ij}}\zeta_jp_i,\qquad
 p_i z_r=q^{\delta_{ir}-\delta_{i,r+1}}z_rp_i.
\end{equation}
Notice that $p_1\cdots p_n$ fixes every ratio $z_r$.  To retain the
central spectral direction in these variables, an eigenfunction must
therefore include a homogeneous character.

Treating $\hbar^{1/2}$ and the $\zeta_i$ initially as indeterminates,
define
\begin{equation}\label{flag:shift-algebra}
 \mathbb K=\C(\hbar^{1/2},\zeta_1,\ldots,\zeta_n),\qquad
 \mathscr D_q=\mathbb K\rtimes\Z^n,
 \qquad [p_i,p_j]=0.
\end{equation}
The cross relations are given by \eqref{flag:kahler-coordinates}.
The fixed $a_i$ may be adjoined to the constant field.  Every element
has a unique finite normal-ordered expansion
$\sum_{\boldsymbol m}c_{\boldsymbol m}(\zeta)p^{\boldsymbol m}$, with
the full shift exponents retained.  In particular, $p_i^p$ is not
identified with $1$.  Put
\begin{equation}\label{flag:central-coordinates}
 Z_i=\zeta_i^p,\qquad \Pi_i=p_i^p.
\end{equation}
Then the center is
\[
 Z(\mathscr D_q)
 =\C(\hbar^{1/2},Z_1,\ldots,Z_n)
       [\Pi_1^{\pm1},\ldots,\Pi_n^{\pm1}].
\]
Indeed, commutation with the $\zeta_i$ forces the shift exponents to
be multiples of $p$, and commutation with the $p_i$ makes their
coefficients invariant under elementary shifts.  Keeping the shifts
formal allows us to record Bloch multipliers; the action on ordinary
rational functions would instead impose $p_i^p=1$.

For $I\subseteq\{1,\ldots,n\}$ write $p_I=\prod_{i\in I}p_i$.
The centered Hamiltonians are
\begin{equation}\label{flag:Hr}
 H_r=\sum_{|I|=r}b_I(\zeta;\hbar)p_I,\qquad
 b_I(\zeta;\hbar)=\hbar^{-r(n-r)/2}
  \prod_{\substack{i\in I\\j\notin I}}
       \frac{\hbar\zeta_i-\zeta_j}{\zeta_i-\zeta_j},
 \qquad 0\leq r\leq n.
\end{equation}
We have $H_0=1$, $H_n=p_1\cdots p_n$, and the $H_r$ commute.
Since their compositions may have poles along $\zeta_i=q^d\zeta_j$,
we first work with rational identities in $\mathscr D_q$.
They specialize on a full orbit whenever $Z_i\ne Z_j$ for $i\ne j$.

We use commuting Cherednik operators $Y_1,\ldots,Y_n$ normalized by
\begin{equation}\label{flag:centered-Y}
 \operatorname{Rad}(e_r(Y_1,\ldots,Y_n))=H_r,
 \qquad 0\leq r\leq n.
\end{equation}
Here $\operatorname{Rad}$ denotes the radial part, defined below.
This convention fixes the common scalar in the $Y_i$ and hence the
spectral normalization of the Adams operation.

For generic $q$, the normalized complete-flag vertices satisfy
\eqref{intro:vertex-equations} by \cite{K21} (Theorem~2.10,
equations~(2.10)--(2.12)).  To compare with the modified vertex
$\widetilde V$ of \cite{D21}, set $u_i=a_i^{-1}$ and
$z_r^{\rm vert}=z_r/\hbar$, so that
$\zeta_r/\zeta_{r+1}=(\hbar/q)z_r^{\rm vert}$.
Its equations read
$\hbar^{-r(n-1)/2}H_r\widetilde V=e_r(a)\widetilde V$.
Thus we choose a character on the lifted shift lattice with
\begin{equation}\label{flag:homogeneous-character}
 p_i\chi=\hbar^{-(n-1)/2}\chi,\qquad V=\chi\widetilde V.
\end{equation}
Then $H_rV=e_r(a)V$.  In particular,
$p_1\cdots p_nV=(a_1\cdots a_n)V$, so the homogeneous character
must be retained when we pass to ratio coordinates.  The corresponding
component-sum construction for equivariant shifts is the qKZ/tRS
correspondence of \cite{KZ21}, reviewed in the introduction.
These generic-$q$ statements concern normalized vertices; at a root
of unity a vertex needs its own definition or regularization.
For the remainder of the operator calculation we return to $q=\omega_p$.

\subsection{The Adams Operation and its Shift Support}
\label{flag:Adams}

We shall express symmetric functions of the powered spectral variables
through the original commuting Hamiltonians.  For this purpose define
the integral Newton--Adams polynomials by
\begin{equation}\label{flag:Adams-definition}
 \mathcal A_{p,r}(e_1(y),\ldots,e_n(y))
       =e_r(y_1^p,\ldots,y_n^p),\qquad 0\leq r\leq n.
\end{equation}
The fundamental theorem of symmetric polynomials determines
$\mathcal A_{p,r}\in\Z[h_1,\ldots,h_n]$ uniquely.  In particular,
$\mathcal A_{p,0}=1$ and $\mathcal A_{p,n}=h_n^p$.
For two variables, $\mathcal A_{p,1}$ is the Newton polynomial
$P_p$ defined in Section~\ref{sec:two}.  For example,
\[
 P_2(u,v)=u^2-2v,\qquad
 P_3(u,v)=u^3-3uv,\qquad
 P_4(u,v)=u^4-4u^2v+2v^2.
\]
These polynomials involve no division by $p$.  Consider the generating
function
\begin{equation}\label{flag:H-generating}
 \mathcal H(u)=\sum_{j=0}^n(-u)^jH_j.
\end{equation}
The elementary identity
\begin{equation}\label{flag:cyclotomic-symmetric}
 \prod_{d=0}^{p-1}\left(\sum_{j=0}^n(-q^du)^je_j(y)\right)
 =\prod_{i=1}^n(1-u^py_i^p)
 =\sum_{r=0}^n(-1)^ru^{pr}e_r(y_1^p,\ldots,y_n^p)
\end{equation}
therefore gives
\begin{equation}\label{flag:cyclotomic-Adams}
 \prod_{d=0}^{p-1}\mathcal H(q^du)
 =\sum_{r=0}^n(-1)^ru^{pr}\mathcal A_{p,r}(H_1,\ldots,H_n).
\end{equation}
Here we have used the commutativity of the $H_j$.  The degrees of
these polynomials give the following bounds on the shift exponents.

\begin{lemma}\label{flag:Adams-support}
The polynomial $\mathcal A_{p,r}$ has weighted degree $pr$ for
$\deg h_j=j$, ordinary degree at most $p$, and coefficient $1$
at $h_r^p$ for $r\geq1$.  Consequently every shift in
$\mathcal A_{p,r}(H)$ satisfies
\begin{equation}\label{flag:box-support}
 0\leq m_i\leq p,\qquad \sum_{i=1}^n m_i=pr.
\end{equation}
\end{lemma}

\begin{proof}
We extract the coefficient of $u^{pr}$ in
\eqref{flag:cyclotomic-symmetric}.  Each of its $p$ factors is linear
in the $h_j$, with $h_0=1$, which gives both degree bounds.
The coefficient of $h_r^p$ in $\mathcal A_{p,r}$ is
\[
 (-1)^{pr-r}q^{rp(p-1)/2}=1,
\]
since $q^{p(p-1)/2}=(-1)^{p-1}$, including when the order is even.
Now each summand of $H_j$ shifts $j$ distinct positions once.
A product of at most $p$ summands therefore shifts any given position
at most $p$ times, and its total shift degree equals the weighted
degree.  Normal ordering affects only the coefficients, so these
bounds also hold in the normal-ordered expansion.
\end{proof}

\subsection{Subset Sequences on the External Orbit}
\label{flag:histories}

The cyclotomic product has a useful combinatorial interpretation.
Each term in its expansion selects a subset at each step, and we can
group the resulting sequences according to their endpoints.  Let
$\mathcal I=(I_0,\ldots,I_{p-1})$ be a sequence of subsets of
$\{1,\ldots,n\}$, with intermediate positions and endpoint given by
\begin{equation}\label{flag:history-endpoints}
 k_d=\sum_{e<d}\boldsymbol1_{I_e},\qquad
 m(\mathcal I)=k_p,\qquad
 \zeta\longmapsto q^{k_d}\zeta.
\end{equation}
Here $m_i$ counts how many times coordinate $i$ is shifted.
To such a sequence we assign a signed weight, and then sum over
sequences with a prescribed endpoint:
\begin{equation}\label{flag:history-weights}
 \begin{aligned}
 w(\mathcal I;\zeta)
   &=(-1)^{|m(\mathcal I)|}
     q^{\sum_{d=0}^{p-1}d|I_d|}
     \prod_{d=0}^{p-1}b_{I_d}(q^{k_d}\zeta;\hbar),\\
 C_m(\zeta)&=\sum_{\mathcal I:\,m(\mathcal I)=m}
                      w(\mathcal I;\zeta),
 \qquad |m|=\sum_i m_i.
 \end{aligned}
\end{equation}
Here normal ordering means moving all rational functions of the
positions to the left and all shifts to the right.  The rule is
\[
 p^k f(\zeta)=f(q^k\zeta)p^k,
 \qquad q^k\zeta=(q^{k_1}\zeta_1,\ldots,q^{k_n}\zeta_n).
\]
After the first $d$ factors, the accumulated shift is $p^{k_d}$.
Moving it past the next coefficient gives
$p^{k_d}b_{I_d}(\zeta;\hbar)
=b_{I_d}(q^{k_d}\zeta;\hbar)p^{k_d}$, which explains the argument
of each factor in \eqref{flag:history-weights}.  We therefore obtain
\begin{equation}\label{flag:history-expansion}
 \prod_{d=0}^{p-1}\mathcal H(q^du)
   =\sum_{m\in\{0,\ldots,p\}^n}u^{|m|}C_m(\zeta)p^m.
\end{equation}
Equation~\eqref{flag:cyclotomic-Adams} implies $C_m=0$ when
$p\nmid|m|$.  To obtain the operator identity we need the stronger
vanishing $C_m=0$ whenever some $0<m_i<p$.
We shall deduce it from Hecke centrality, and then give a direct
calculation at $p=2$.  The sequences here describe external shifts;
internal roots and Bethe equations do not enter their weights.

\subsection{Root-of-Unity Centrality}
\label{flag:centrality-subsection}

Let us recall the part of the type-$\GL(n)$ double affine Hecke algebra
which is needed for this calculation; see
\cite{VV10} (Sections~3.1--3.4).  There are two separately commuting
families of invertible elements.  The coordinates $X_1,\ldots,X_n$
act by multiplication by $\zeta_1,\ldots,\zeta_n$, and their duals
$Y_1,\ldots,Y_n$ are the commuting Cherednik operators.
The latter are constructed from translations in the
affine Hecke algebra and act by rational difference--reflection
operators, combining multiplicative shifts with permutations.
We use the centered $Y_i$ of \eqref{flag:centered-Y}.

The Hecke generators $T_i$ deform the simple reflections $s_i$.
The finite reflections exchange adjacent coordinates, while the
affine reflection also contains a $q$-shift.  Besides the braid
relations, these generators satisfy $(T_i-v)(T_i+v^{-1})=0$, where
$v^2=\hbar$.  Their commutation with coordinate functions is given
by the Bernstein relation, in a compatible simple-root convention,
\begin{equation}\label{flag:bernstein-relation}
 T_i f(X)-(s_i f)(X)T_i
   =(v-v^{-1})\frac{f(X)-(s_i f)(X)}{1-X^{\alpha_i}},
\end{equation}
where $\alpha_i$ is the corresponding simple affine root.
Its monomial $X^{\alpha_i}$ includes the affine parameter when
appropriate.  Thus a function invariant under affine reflections
commutes with these generators.  The length-zero elements of the
extended affine Weyl group act by coordinate substitution.
Fourier duality exchanges the $X$ and $Y$ families, up to the common
normalization scalar and the conventional inversions and parameter
substitutions.  It carries central symmetric functions of one family
to those of the other.

In the localized Dunkl realization these identities hold in the
formal algebra $\mathscr D_q\rtimes S_n$.  Writing every permutation
to the right, we define the radial part by
\begin{equation}\label{flag:radial-part}
 L=\sum_{w\in S_n}L_w w
 \quad\longmapsto\quad
 \operatorname{Rad}(L)=\sum_{w\in S_n}L_w,
 \qquad L_w\in\mathscr D_q.
\end{equation}
The permutations are removed as in the action on symmetric functions,
but every lattice shift is retained as a formal symbol.  On symmetric
functions of the commuting $Y_i$, radial restriction is multiplicative
and satisfies $\operatorname{Rad}(e_r(Y))=H_r$.
We only use multiplicativity on this subalgebra, not on arbitrary
difference--reflection operators.  A symmetric function $g(\zeta)$
commutes with every permutation, giving
\[
 \operatorname{Rad}([L,g])=[\operatorname{Rad}(L),g].
\]
Thus a central element has a radial part which commutes with
symmetric coordinate functions.

Here $q$ is the actual multiplicative shift of the coordinates.
The convention of \cite{VV10} is $X_\delta=q_{\rm VV}^{-2}$, so
its parameter is to be expressed in terms of this shift.  All our
calculations take place in the formal realization, where $p_i^p$
remains a distinct symbol; faithfulness on ordinary Laurent
polynomials at a root of unity is not required.
The affine invariance of symmetric functions of $X_1^p,\ldots,X_n^p$
now gives the commutation relation needed for the Adams operators.

\begin{lemma}\label{flag:centrality}
The operator $\mathcal A_{p,r}(H)$ commutes with multiplication by
every symmetric rational function of $\zeta_1,\ldots,\zeta_n$.
\end{lemma}

\begin{proof}
First assume that the Hecke parameter is generic and consider a
symmetric Laurent polynomial $f$ in $X_1^p,\ldots,X_n^p$.
It is invariant under the extended affine Weyl action.  Indeed,
permutations preserve its symmetry, while translations multiply
the $X_i$ by powers of $q$, whose $p$th powers are one.
The right side of \eqref{flag:bernstein-relation} vanishes, so $f$
commutes with the affine Hecke generators.  The conjugation relation
for length-zero elements gives commutation with those elements as
well.  Since $f$ also commutes with the coordinate generators, it
is central.

Applying Fourier duality shows that every symmetric Laurent polynomial
in $Y_1^p,\ldots,Y_n^p$ is central.  The common normalization scalar,
inversions, and parameter involution do not change this conclusion.
Using the defining identity of the Adams polynomial and
multiplicativity of radial restriction on symmetric
functions of $Y$, we obtain
\[
 \begin{aligned}
 \operatorname{Rad}\bigl(e_r(Y_1^p,\ldots,Y_n^p)\bigr)
   &=\operatorname{Rad}\bigl(\mathcal A_{p,r}(e_1(Y),\ldots,e_n(Y))\bigr)\\
   &=\mathcal A_{p,r}(H).
 \end{aligned}
\]
The required commutation follows by applying
\eqref{flag:radial-part} to the commutator with a symmetric
coordinate function.

The argument has taken place in the formal difference--reflection
algebra, so distinct lattice shifts remain distinct upon specialization.
To specialize the Hecke parameter, write the commutator using
\eqref{flag:Hr}.  Its coefficients are Laurent polynomial in
$\hbar^{1/2}$, with denominators depending only on positions and
their $q$-translates.  The generic identity therefore extends to
every nonzero $\hbar^{1/2}$ wherever these coefficients are defined.
\end{proof}

\begin{remark}\label{flag:order-remark}
The restrictions on root orders in some geometric results of
\cite{VV10} are not needed here.  Our proof uses the Bernstein
relations, affine invariance, and Fourier duality, and applies to
every $p$, including composite and even orders.  The explicit
rational formulas for the $H_r$ remove any denominators introduced
by an auxiliary Hecke symmetrizer.
\end{remark}

\subsection{Cancellation of All Intermediate Shifts}
\label{flag:operator-identity}

We can now combine the commutation relation with the bounds of
Lemma~\ref{flag:Adams-support}.  Introduce the operators
\begin{equation}\label{flag:powered-Hr}
 H_r^{[p]}=\hbar^{-pr(n-r)/2}\sum_{|I|=r}
  \prod_{\substack{i\in I\\j\notin I}}
       \frac{\hbar^pZ_i-Z_j}{Z_i-Z_j}\,\Pi_I,
 \qquad \Pi_I=\prod_{i\in I}\Pi_i.
\end{equation}
These belong to $Z(\mathscr D_q)$.  The following theorem identifies
them with the Adams polynomials in the original Hamiltonians.

\begin{theorem}\label{flag:frobenius}
For every $n,p\geq1$ and every primitive $p$th root $q$,
\begin{equation}\label{flag:Adams-operator}
 \mathcal A_{p,r}(H_1,\ldots,H_n)=H_r^{[p]},
 \qquad 0\leq r\leq n.
\end{equation}
Equivalently,
\begin{equation}\label{flag:Frobenius-factorization}
 \prod_{d=0}^{p-1}\mathcal H(q^du)
   =\sum_{r=0}^n(-1)^ru^{pr}H_r^{[p]}.
\end{equation}
The surviving coefficients are obtained by the Frobenius action
on coordinates, coupling, and shifts:
\[
 (\zeta_i,\hbar^{1/2},p_i)\longmapsto
 (\zeta_i^p,\hbar^{p/2},p_i^p),
\]
with no extra parity sign.
\end{theorem}

\begin{proof}
For $r=0$ the identity is immediate.  Let $r\geq1$ and write
$F_r=\mathcal A_{p,r}(H)=\sum_{\boldsymbol m}
c_{\boldsymbol m}(\zeta)p^{\boldsymbol m}$.
By Lemma~\ref{flag:centrality}, we have
\[
 0=[F_r,e_1(\zeta)]
  =\sum_{\boldsymbol m}c_{\boldsymbol m}(\zeta)
    \left(\sum_i(q^{m_i}-1)\zeta_i\right)p^{\boldsymbol m}.
\]
The normal-ordered expansion is unique and the $\zeta_i$ are
algebraically independent.  Hence $c_{\boldsymbol m}\ne0$ requires
$p\mid m_i$ for every $i$.  Together with
Lemma~\ref{flag:Adams-support}, this leaves only the shifts
\begin{equation}\label{flag:surviving-shifts}
 p^{\boldsymbol m}=p_I^p,\qquad |I|=r.
\end{equation}

It remains to compute the coefficient for a fixed $I$.
Every coordinate in $I$ must be shifted $p$ times, so a contributing
monomial in the $H_j$ has exactly $p$ nonconstant factors.
Each chosen summand must shift every member of $I$ and no coordinate
outside $I$.  Consequently each factor is $H_r$, and the only Adams
monomial which contributes is $H_r^p$, with coefficient one.
Its coefficient at $p_I^p$ is the product around the full orbit.
Applying $\prod_{d=0}^{p-1}(1-q^dw)=1-w^p$ to each pair factor,
we obtain
\begin{equation}\label{flag:orbit-norm}
 \prod_{d=0}^{p-1}b_I(q^{d\boldsymbol1_I}\zeta;\hbar)
 =\hbar^{-pr(n-r)/2}
   \prod_{\substack{i\in I\\j\notin I}}
       \frac{\hbar^p\zeta_i^p-\zeta_j^p}{\zeta_i^p-\zeta_j^p}.
\end{equation}
The leading cyclotomic factors cancel between numerator and
denominator, also at even orders.  This proves
\eqref{flag:Adams-operator}, and the generating form follows from
\eqref{flag:cyclotomic-Adams}.  As in the centrality lemma, clearing
the position denominators extends the result from generic Hecke
parameter to every nonzero $\hbar^{1/2}$.
\end{proof}

\begin{corollary}\label{flag:eigenfunction-frobenius}
Let $n,p\geq1$, let $q$ be a primitive $p$th root of unity, and let
$V$ be a nonzero joint solution on a shift-stable locus where the
operators are defined; it suffices to require
$\zeta_i^p\ne\zeta_j^p$ for $i\ne j$.
If $V$ satisfies, for fixed spectral parameters $a_i$,
\[
 H_rV=e_r(a_1,\ldots,a_n)V,\qquad 1\leq r\leq n,
\]
then
\begin{equation}\label{flag:eigenfunction-powered-equations}
 H_r^{[p]}V
 =\mathcal A_{p,r}(H_1,\ldots,H_n)V
 =e_r(a_1^p,\ldots,a_n^p)V,
 \qquad 1\leq r\leq n.
\end{equation}
\end{corollary}

\begin{proof}
Apply the Adams polynomial to the common eigenfunction.  Since the
Hamiltonians commute, we get the same polynomial in their eigenvalues:
\[
 \mathcal A_{p,r}(H)V
 =\mathcal A_{p,r}(e_1(a),\ldots,e_n(a))V
 =e_r(a_1^p,\ldots,a_n^p)V.
\]
The left side is $H_r^{[p]}V$ by Theorem~\ref{flag:frobenius},
which proves the assertion.
\end{proof}

The corollary applies to any joint solution defined at the chosen root
of unity, without imposing Bethe equations or fixing Bloch multipliers.
Specializing an infinite vertex series from generic $q$ remains a
separate question.

\begin{corollary}
\label{flag:history-cancellation}
The sums \eqref{flag:history-weights} satisfy
\begin{equation}\label{flag:history-identity}
 C_m(\zeta)=
 \begin{cases}
 (-1)^{|I|}\hbar^{-p|I|(n-|I|)/2}
 \displaystyle\prod_{i\in I,\,j\notin I}
       \dfrac{\hbar^p\zeta_i^p-\zeta_j^p}{\zeta_i^p-\zeta_j^p},
       &m=p\boldsymbol1_I,\\[6pt]
 0,&\text{otherwise}.
 \end{cases}
\end{equation}
For $m=p\boldsymbol1_I$, the sole subset sequence is
$I_0=\cdots=I_{p-1}=I$.
\end{corollary}

\begin{proof}
Comparing \eqref{flag:history-expansion} with
\eqref{flag:Frobenius-factorization} gives the stated coefficients.
To reach $p\boldsymbol1_I$, each coordinate in $I$ must occur at all
$p$ steps, while every other coordinate is omitted.  The weight of
this unique sequence is the orbit norm \eqref{flag:orbit-norm}
multiplied by
$(-1)^{p|I|}q^{|I|p(p-1)/2}=(-1)^{|I|}$.
\end{proof}

It is useful to compare this calculation with cyclic rotation of
the subset sequences.  Let
$m=m(\mathcal I)$, $R\mathcal I=(I_1,\ldots,I_{p-1},I_0)$, and
$s=\boldsymbol1_{I_0}$.
The weights give
\begin{equation}\label{flag:history-rotation}
 w(R\mathcal I;q^s\zeta)b_{I_0}(\zeta)
 =q^{-|m|}w(\mathcal I;\zeta)b_{I_0}(q^m\zeta).
\end{equation}
Thus rotation changes both the starting point and the argument of the
last coefficient.  At $|m|=pr$ the root-of-unity factor is already
one, so cyclic rotation alone does not account for cancellation of
individual partial shifts.  The general proof above uses centrality.
At order two we can instead prove cancellation directly, as follows.

\subsection{A Direct Cancellation at Order Two}
\label{flag:examples}

Consider $q=-1$.  We shall obtain the cancellation in every rank
from a signed bipartition identity, proved by pairing collision
residues.  First write the coefficient of a subset as a product
of pair weights:
\begin{equation}\label{flag:pair-weight}
 f_\hbar(v,w)=\hbar^{-1/2}\frac{\hbar v-w}{v-w},
 \qquad
 b_I(\zeta;\hbar)=\prod_{i\in I,\,j\notin I}
                         f_\hbar(\zeta_i,\zeta_j).
\end{equation}

\begin{lemma}
\label{flag:bipartition}
For $N\geq1$ and an indeterminate $c$, the rational function
\begin{equation}\label{flag:bipartition-sum}
 S_N(v;c)=\sum_{L\subseteq\{1,\ldots,N\}}(-1)^{|L|}
  \prod_{i\in L,\,j\notin L}
       \left(1+c\frac{v_iv_j}{v_i^2-v_j^2}\right)
\end{equation}
vanishes identically.  Set $S_0=1$.
\end{lemma}

\begin{proof}
We argue by induction on $N$, with $S_1=0$ as the initial case.
Regard $S_N$ as a rational function of $v_1$, keeping the other
variables generic.  Each pair contributes at most one crossing
factor in a summand, so the possible finite poles are simple and
lie at $v_1=\pm v_j$.  At $v_1=v_2$, pair terms in which $1$ and
$2$ belong to opposite sides by exchanging their memberships.
Their signs and all other factors agree on this divisor, whereas
the residues of the crossing factors have opposite signs.

For the divisor $v_1=-v_2$, put
\[
 a_k=-c\frac{v_2v_k}{v_2^2-v_k^2},\qquad
 P_\pm=\prod_{k=3}^N(1\pm a_k).
\]
When $1\in L$ and $2\notin L$, their crossing factors with the
remaining vertices have product $P_+$ on this divisor, independent
of the memberships of those vertices.  The opposite assignment
gives $P_-$.  Since the two singular pair factors have opposite
residues, summing over the remaining memberships gives a residue
proportional to
\[
 (P_+-P_-)S_{N-2}(v_3,\ldots,v_N;c).
\]
For $N=2$ the first factor vanishes, and for $N>2$ the induction
hypothesis applies.  The same calculation works for every $j$,
so there are no finite poles.  Finally, as $v_1\to\infty$, every
factor involving vertex $1$ tends to one.  The two choices of its
membership then cancel.  Thus $S_N$ has no finite poles and tends
to zero at infinity, proving that it vanishes identically.
\end{proof}

\begin{proposition}
\label{flag:direct-order-two}
For every $n$, the identity \eqref{flag:Frobenius-factorization} at
$q=-1$ follows from Lemma~\ref{flag:bipartition}, independently of
Hecke centrality.
\end{proposition}

\begin{proof}
For an endpoint $m\in\{0,1,2\}^n$, partition the coordinates into
$D=\{i:m_i=2\}$, $E=\{i:m_i=1\}$, and $F=\{i:m_i=0\}$.
The subset sequences with this endpoint are precisely
$I_0=D\cup L$, $I_1=D\cup(E\setminus L)$ with $L\subseteq E$.
The pair factors involving $D$ or $F$ do not depend on $L$ and
give the common product
\begin{equation}\label{flag:bipartition-common-factor}
 \begin{aligned}
 K_{D,E,F}={}&
 \prod_{d\in D,\,f\in F}
    f_\hbar(\zeta_d,\zeta_f)f_\hbar(-\zeta_d,\zeta_f)\\
 &\mathrel{}\cdot\prod_{d\in D,\,e\in E}f_\hbar(\zeta_d,\zeta_e)
 \prod_{e\in E,\,f\in F}f_\hbar(\zeta_e,\zeta_f).
 \end{aligned}
\end{equation}
For $i\in L$, $j\in E\setminus L$, the remaining pair contributes
\begin{equation}\label{flag:bipartition-pair}
 f_\hbar(\zeta_i,\zeta_j)f_\hbar(\zeta_j,-\zeta_i)
 =1+(\hbar-\hbar^{-1})
          \frac{\zeta_i\zeta_j}{\zeta_i^2-\zeta_j^2}.
\end{equation}
The sign in $\mathcal H(u)\mathcal H(-u)$ is
$(-1)^{|I_0|}=(-1)^{|D|+|L|}$.  Collecting the above factors, we
find that its coefficient at $u^{|m|}p^m$ is
\begin{equation}\label{flag:bipartition-coefficient}
 C_m=(-1)^{|D|}K_{D,E,F}
       S_{|E|}((\zeta_e)_{e\in E};\hbar-\hbar^{-1}).
\end{equation}
The lemma makes this coefficient zero when $E\ne\varnothing$.
For $E=\varnothing$, the two-step orbit norm gives
$K_{D,\varnothing,F}=b_D(\zeta^2;\hbar^2)$.
We have therefore recovered exactly the coefficients of
\eqref{flag:Frobenius-factorization} at $p=2$.
\end{proof}

Extracting coefficients from $\mathcal H(u)\mathcal H(-u)$, we can
write the order-two identity as
\begin{equation}\label{flag:order-two-all-r}
 H_r^{[2]}=H_r^2+
  2\sum_{j=1}^{\min(r,n-r)}(-1)^jH_{r-j}H_{r+j}.
\end{equation}
For example, for $T^*\Fl_3$ we obtain
\begin{equation}\label{flag:n3-k2}
 \begin{aligned}
 H_1^{[2]}&=H_1^2-2H_2,&
 H_2^{[2]}&=H_2^2-2H_1H_3,&
 H_3^{[2]}&=H_3^2.
 \end{aligned}
\end{equation}
Only the shifts $p_1^2$, $p_2^2$, and $p_3^2$ survive in the first
identity.  In the second they are $p_1^2p_2^2$, $p_1^2p_3^2$, and
$p_2^2p_3^2$, with the orbit-norm coefficients of
\eqref{flag:powered-Hr}.  The calculation takes place in
characteristic zero, so the coefficient $2$ is retained.

\subsection{Bloch Multipliers and the Classical Spectral Fiber}
\label{flag:Bloch}

We shall now interpret the powered eigenvalue equations in terms of
the classical tRS spectral fiber.  Let $V$ be a nonzero section on a
lifted shift orbit, with full-cycle multipliers given by
\begin{equation}\label{flag:Bloch-multipliers}
 p_i^pV=\lambda_i V,\qquad 1\leq i\leq n.
\end{equation}
Here the $\lambda_i$ are nonzero scalars on a fixed orbit.
For a meromorphic family, they are functions invariant under all
elementary shifts.

For a symplectic parameter $c$, work over $\C(c^{1/2})$ and consider
\begin{equation}\label{flag:classical-Hr}
 h_r(Z,\pi;c)=c^{-r(n-r)/2}\sum_{|I|=r}
  \prod_{\substack{i\in I\\j\notin I}}
       \frac{cZ_i-Z_j}{Z_i-Z_j}\prod_{i\in I}\pi_i,
 \qquad h_0=1.
\end{equation}
These are the centered classical tRS Hamiltonians with usual coupling $c^{-1}$.
We shall apply this formula with $c=\hbar^p$.

\begin{proposition}\label{flag:powered-Bloch-criterion}
Suppose \eqref{flag:Bloch-multipliers} holds and $Z_i\ne Z_j$
for $i\ne j$.  Then
\begin{equation}\label{flag:powered-eigenfunction}
 H_r^{[p]}V=\mathcal A_{p,r}(H)V=e_r(a_1^p,\ldots,a_n^p)V,
 \qquad 1\leq r\leq n,
\end{equation}
if and only if
\begin{equation}\label{flag:powered-Bloch-levelset}
 h_r(Z,\lambda;\hbar^p)=e_r(a_1^p,\ldots,a_n^p),
 \qquad 1\leq r\leq n.
\end{equation}
\end{proposition}

\begin{proof}
Substitute the full-cycle eigenvalues \eqref{flag:Bloch-multipliers}
into \eqref{flag:powered-Hr}.  This gives
\[
 H_r^{[p]}V=h_r(Z,\lambda;\hbar^p)V.
\]
By Theorem~\ref{flag:frobenius}, the left side equals
$\mathcal A_{p,r}(H)V$.  The scalar coefficient is constant on a
fixed coordinate orbit, and $V$ has a nonzero value there, proving
necessity.  Sufficiency follows from the same equality.
For a meromorphic family the necessary identity first holds where
$V$ is nonzero and then extends meromorphically.
\end{proof}

In particular, once \eqref{flag:powered-Bloch-levelset} holds,
every section in the same Bloch sector satisfies the powered
equations.  The ordinary joint eigenvalue problem imposes the
further conditions
\begin{equation}\label{flag:joint-eigenfunction}
 H_rV=e_r(a_1,\ldots,a_n)V,\qquad 1\leq r\leq n.
\end{equation}
These give the following special case.

\begin{corollary}\label{flag:Bloch-corollary}
Under \eqref{flag:joint-eigenfunction} and
\eqref{flag:Bloch-multipliers},
\begin{equation}\label{flag:bloch-levelset}
 h_r(Z,\lambda;\hbar^p)=e_r(a_1^p,\ldots,a_n^p),
 \qquad 1\leq r\leq n.
\end{equation}
\end{corollary}

\begin{proof}
Acting on $V$ with \eqref{flag:Adams-operator}, the left side gives
$\mathcal A_{p,r}(e_1(a),\ldots,e_n(a))=e_r(a^p)$, while the right
side gives $h_r(Z,\lambda;\hbar^p)$.
We can cancel $V$ at a nonzero value.  Since the coefficients are
shift invariant, this proves the assertion on the fixed orbit.
In a meromorphic family the identity extends from the nonvanishing
locus.
\end{proof}

The same statement can be expressed through the spectral polynomial
\begin{equation}\label{flag:classical-spectral-polynomial}
 \mathscr P(v;Z,\pi,c)=\sum_{r=0}^n(-1)^rh_r(Z,\pi;c)v^{n-r}
\end{equation}
which satisfies
\begin{equation}\label{flag:spectral-polynomial-identity}
 \mathscr P(v;Z,\lambda,\hbar^p)
       =\prod_{i=1}^n(v-a_i^p).
\end{equation}
This is the classical spectral fiber appearing in both the criterion
and its ordinary-eigenfunction special case.  For finite sums, Bethe
closure gives the full-cycle multipliers explicitly.  The classical
tRS/Bethe correspondence of \cite{KPSZ21} (Theorems~4.6 and~4.9),
applied at powered parameters, places them on
\eqref{flag:powered-Bloch-levelset}.  It follows that every admissible
rational insertion satisfies the powered equations, as we shall prove
in Theorem~\ref{cyc:finite-frobenius}; a nonzero sum is a joint
eigenfunction of these operators.  This gives a second application
of the operator identity, alongside its direct consequence for an
ordinary joint eigenfunction in
Corollary~\ref{flag:eigenfunction-frobenius}.

Recall that the momenta $\lambda_i$ measure full-cycle transport.
In general they are not the $p$th powers of momenta evaluated at the
original parameters.  This is the same distinction that appears for
K\"ahler monodromy: its matrices are multiplied at successively
shifted positions, rather than taking the power of one matrix at a
fixed position.  The classical fiber of the cyclic sums is described
by these full-cycle multipliers.

The homogeneous normalization has a simple effect on the spectrum.
Multiplying a section by $b^{m_1+\cdots+m_n}$ on the lifted lattice
replaces the spectral parameters by $ba_i$, multiplies the $r$th
energy by $b^r$, and multiplies each Bloch multiplier by $b^p$.

\begin{proposition}
\label{flag:generic-degree}
For generic distinct nonzero $Z_i$, generic $c\ne0$, and generic
$s_1,\ldots,s_n$ with $s_n\ne0$, the equations
\begin{equation}\label{flag:classical-fiber}
 h_r(Z,\pi;c)=s_r,\qquad 1\leq r\leq n,
\end{equation}
have exactly $n!$ distinct solutions in $(\C^*)^n$.  In particular,
it holds for generic spectral data $s_r=e_r(a_1^p,\ldots,a_n^p)$.
\end{proposition}

\begin{proof}
Since $h_r$ is homogeneous of degree $r$ in $\pi$ and
$h_n=\pi_1\cdots\pi_n$, the equations homogenize in $\mathbf P^n$ as
\begin{equation}\label{flag:homogenized-fiber}
 h_r(Z,\pi;c)-s_rw^r=0.
\end{equation}
At $c^{1/2}=1$ we recover the elementary symmetric equations.
They have no point at infinity: the conditions $e_r(\pi)=0$ for
every $r$ would force all $\pi_i=0$.  Taking
$s_r=e_r(d_1,\ldots,d_n)$ with distinct nonzero $d_i$, we obtain
their $n!$ permutations as solutions.  At each solution the Jacobian
is, up to sign, the nonzero Vandermonde product.

Exclude the position diagonals and $c=0$.  The projective incidence
variety is then proper over the parameter space, so the images of
its infinity locus and of the locus where the projective Jacobian
has rank less than $n$ are closed.  Their complement contains the
point constructed above.  Over this open set a fiber is projective
and lies in an affine chart; all its points are nonsingular of
dimension zero.  Thus it is finite and reduced, and B\'ezout's
theorem gives $1\cdot2\cdots n=n!$ points.  Finally,
$h_n=s_n\ne0$ places all these points in the torus.

The map taking $(a_1,\ldots,a_n)\in(\C^*)^n$ to
$e_r(a_1^p,\ldots,a_n^p)$ is dominant, so the same open condition
holds for generic spectral data.
\end{proof}

We have counted the classical momentum points of the spectral fiber.
For the finite functionals associated with closed Bethe orbits, these
momenta are the full-cycle multipliers and give the powered equations
of Theorem~\ref{cyc:finite-frobenius}.  We now turn to the explicit
construction of the internal sums and their multipliers.

\subsection{The Flag Phase and K\"ahler Transport}
\label{cyc:construction}

Let us describe the finite sums associated with the flag vertex
phase.  The internal shift ratios determine their weights, while
K\"ahler shifts give the Bloch multipliers and the corresponding
tRS momenta.

We denote by $\mathcal V_r$ the rank-$r$ tautological quotient
bundles on $\Fl_n$ and their pullbacks to $T^*\Fl_n$.
They fit into the sequence of quotient maps
\[
 \mathcal W=\C^n\otimes\mathcal O
       \longrightarrow\mathcal V_{n-1}
       \longrightarrow\cdots\longrightarrow\mathcal V_1.
\]
Using the splitting notation for their equivariant $K$-classes, we write
\[
 [\mathcal V_r]=\sum_{u=1}^r x_{r,u},\qquad x_{n,i}=a_i,\qquad
 L_r=\prod_{u=1}^r x_{r,u},\qquad L_0=1,\quad L_n=a_1\cdots a_n.
\]
Here $L_r$ represents $\det\mathcal V_r$.  Its evaluation at
Bethe roots gives the corresponding quantum eigenvalue.  We use
$\mathcal V_r$ for the bundles and $V$ for the vertex function.
In these variables the polarization takes the form
\begin{equation}\label{cyc:polarization}
 P=\sum_{r=1}^{n-1}\sum_{u=1}^r\sum_{v=1}^{r+1}
       \frac{x_{r,u}}{x_{r+1,v}}
   -\sum_{r=1}^{n-1}\sum_{u,v=1}^r\frac{x_{r,u}}{x_{r,v}},
 \qquad TX=P+\hbar^{-1}P^*.
\end{equation}
The contribution of $L_r$ to $\det P$ has exponent
$(r+1)-(r-1)=2$.  The shifted geometric coordinates are therefore
$z_{\#,r}=z_r/\hbar$.  For $|q|<1$ let
\[
 \varphi(w,q)=\prod_{d\geq0}(1-q^dw),\qquad
 B_q(w)=\frac{\varphi(qw,q)}{\varphi(\hbar w,q)}.
\]
If $P=\sum_w m_w w$, the phase of the vertex integral is given by
\begin{equation}\label{cyc:flag-phase}
 \mathscr F(\boldsymbol z,\boldsymbol a,\boldsymbol x;q)
 =\prod_{r=1}^{n-1}
   \exp\left(\frac{\log(z_r/\hbar)\log L_r}{\log q}\right)
   \prod_w B_q(w)^{m_w}.
\end{equation}
The rational shift ratios of this phase remain the data of our
construction at roots of unity.  They define the finite sums
below without evaluating the infinite products in
\eqref{cyc:flag-phase} at such values of $q$.

For generic $q$ we can write the vertex functional with insertion
$f$ as
\[
 \mathcal S_\Gamma[f]
 =\mathcal N_\Gamma\,\chi(\boldsymbol\zeta)
   \int_\Gamma\mathscr F(\boldsymbol z,\boldsymbol a,\boldsymbol x;q)
     f(\boldsymbol x)
     \prod_{\alpha=1}^{R}\frac{\dd x_\alpha}{2\pi\mathrm i x_\alpha},
 \qquad R=\frac{n(n-1)}2.
\]
Here $\Gamma$ is a cycle or residue prescription selecting a
fixed-point branch.  The factor $\chi$ contains the homogeneous
character and its centering, which we describe below, and
$\mathcal N_\Gamma$ sets the leading coefficient of the bare
fixed-point vertex equal to one.  We keep the same cycle and
normalization for all insertions.  In particular,
$\mathcal S_\Gamma[1]$ is the vertex with unit insertion.

Let $T_{r,u}$ be the shift which multiplies $x_{r,u}$ by $q$.
Writing $x=x_{r,u}$, we obtain the internal transport ratio
\begin{equation}\label{cyc:internal-ratios}
\begin{split}
 \Phi_{r,u}={}&\frac{z_r}{\hbar}
 \prod_{v=1}^{r+1}\frac{1-\hbar x/x_{r+1,v}}{1-qx/x_{r+1,v}}
 \prod_{v=1}^{r-1}\frac{1-x_{r-1,v}/x}{1-\hbar x_{r-1,v}/(qx)}\\
 &\times\prod_{\substack{1\leq v\leq r\\v\ne u}}
 \frac{(1-qx/x_{r,v})(1-\hbar x_{r,v}/(qx))}
      {(1-\hbar x/x_{r,v})(1-x_{r,v}/x)}.
\end{split}
\end{equation}
Indeed, this expression follows by applying
$B_q(qw)/B_q(w)=(1-\hbar w)/(1-qw)$ to the weights of $P$.
Here and below an empty product equals one.

We next consider a shift of the geometric K\"ahler coordinate
$z_r$.  It multiplies the phase by $L_r$.  The homogeneous
extension contains $\exp(\log\zeta_n\log L_n/\log q)$ and hence
has central character $L_n$.  In the coordinates
\eqref{intro:kahler-coordinates}, multiplication by a character
with $p_i$-multiplier $\hbar^{(n+1)/2-i}$ gives
\begin{equation}\label{cyc:external-ratios}
 \rho_i(\boldsymbol x)=
   \hbar^{(n+1)/2-i}\frac{L_i}{L_{i-1}},\qquad 1\leq i\leq n.
\end{equation}
The shift $p_i$ acts forward on $z_i$ and backward on $z_{i-1}$
whenever these coordinates exist, so the simultaneous shift has
character $\prod_i\rho_i=a_1\cdots a_n$.
This expression also has a direct interpretation in quantum
$K$-theory.  Let $\widehat\tau$ be the quantum tautological class
associated with an insertion $\tau$, and denote the quantum
product by $\star$.  We write $\star\,\gamma$ for quantum
multiplication by a class $\gamma$.  In our normalization,
\cite{KPSZ21} (Theorem~4.6, equation~(59)) identifies the tRS
momentum on the Bethe level set with the operator
\begin{equation}\label{cyc:quantum-momentum}
 \mathsf P_d^{\mathrm{qK}}
 =\hbar^{(n+1)/2-d}\star\!\left(
   \widehat{\det\mathcal V_d}\star
   \widehat{\det(\mathcal V_{d-1}^{\vee})}\right),
 \qquad 1\leq d\leq n,
\end{equation}
where $\mathcal V_0=0$ and $\mathcal V_n=\mathcal W$.
According to \cite{KPSZ21} (Theorem~3.4), the two quantum classes
have Bethe eigenvalues $L_d$ and $L_{d-1}^{-1}$, respectively.
Thus the eigenvalue of $\mathsf P_d^{\mathrm{qK}}$ is precisely
$\rho_d$.  Notice that this identification uses the quantum
classes and their quantum product: on the generic semisimple
locus the dual determinant represents the quantum inverse.
The cotangent weight in \cite{KPSZ21} is $\hbar^{-1}$ in our
notation.  Dividing its momenta by the common factor
$\hbar^{-n/2}$ gives the spectral parameters $a_1,\ldots,a_n$.
When $\mathcal B_{p,\alpha}=1$, the roots $\boldsymbol x^p$ solve the Bethe
equations at $(\boldsymbol z^p,\boldsymbol a^p,\hbar^p)$.
The momenta \eqref{cyc:flag-momenta} below are therefore quantum
multiplication eigenvalues at these powered parameters.  This
does not require $\boldsymbol x$ to solve the Bethe equations at
the original parameters.

The internal and external ratios obey the flatness identities
\begin{align}\label{cyc:flatness}
 (T_\beta\Phi_\alpha)\Phi_\beta
   &=(T_\alpha\Phi_\beta)\Phi_\alpha,\notag\\
 (p_j\rho_i)\rho_j&=(p_i\rho_j)\rho_i,\qquad
 (p_i\Phi_\alpha)\rho_i=(T_\alpha\rho_i)\Phi_\alpha.
\end{align}
The first relation follows from commuting the internal shifts of
the phase.  The second follows from the independence of $\rho_i$
of $\boldsymbol z$.  For $\alpha=(r,u)$, the two sides of the
mixed relation acquire the same factor
$q^{\delta_{ir}-\delta_{i,r+1}}$.

Now set $q=\omega_p$ and $R=n(n-1)/2$.  Independently shifting
the nonzero variables $x_{r,u}$ gives an internal orbit with
$p^R$ points.  We assume that it avoids zeros and poles of the
phase ratios and poles of the rational insertions.  Define
\begin{equation}\label{cyc:norms}
 \mathcal B_{p,\alpha}=\prod_{d=0}^{p-1}T_\alpha^d\Phi_\alpha,
 \qquad \lambda_i=\rho_i^p.
\end{equation}
For one internal variable, $\mathcal B_{p,\alpha}$ is the
product $\mathcal B_p$ of Section~\ref{sec:two}.
These norms are invariant under both internal and K\"ahler shifts.
For an internal shift this follows from the telescoping product
\[
 \frac{T_\beta \mathcal B_{p,\alpha}}{\mathcal B_{p,\alpha}}
 =\prod_{d=0}^{p-1}
   \frac{T_\alpha^{d+1}\Phi_\beta}{T_\alpha^d\Phi_\beta}
 =\frac{T_\alpha^p\Phi_\beta}{\Phi_\beta}=1.
\]
Here we used flatness.  The mixed relation in
\eqref{cyc:flatness} gives the same calculation for a K\"ahler
shift.  Using $\prod_{d=0}^{p-1}(1-q^dw)=1-w^p$, we find
\begin{equation}\label{cyc:powered-Bethe}
 \mathcal B_{p,(r,u)}=\Phi_{r,u}\big|_{q=1}
       (\boldsymbol z^p,\boldsymbol a^p,\boldsymbol x^p;\hbar^p).
\end{equation}
Hence $\mathcal B_{p,(r,u)}=1$ are the Bethe equations at the powered
parameters.  The corresponding K\"ahler multipliers read
\begin{equation}\label{cyc:flag-momenta}
 \lambda_i=\hbar^{p((n+1)/2-i)}
           \left(\frac{L_i}{L_{i-1}}\right)^p,
 \qquad
 \prod_{i=1}^r\lambda_i=\hbar^{pr(n-r)/2}L_r^p.
\end{equation}
In particular, the monodromy of the geometric shift
$z_r\mapsto qz_r$ has eigenvalue $L_r^p$.  The powers of
$\hbar$ in \eqref{cyc:flag-momenta} come from the centered
homogeneous gauge.

\subsection{Bethe Equations as Closure Conditions}

We shall now interpret the Bethe equations as closure conditions
for the finite internal grid.  They make the phase weights
periodic and hence allow twisted total differences to cancel
under summation.

The summation variables are the multiplicative tautological roots
$x_{r,u}$, $1\leq u\leq r<n$, of the phase
\eqref{cyc:flag-phase}.  There are $R=n(n-1)/2$ such variables,
which we also denote by $x_\alpha$ for $\alpha=(r,u)$.
Let $\boldsymbol x^0=(x_{r,u}^0)$ be a point of this auxiliary
torus with all coordinates nonzero.  The superscript $0$ records
the initial point of the orbit.  These variables represent
tautological $K$-classes; they are not local coordinates on
$T^*\Fl_n$.

For the internal shifts we fix the K\"ahler parameters
$\boldsymbol z$, or equivalently the chosen homogeneous
parameters $\boldsymbol\zeta$, together with the equivariant
parameters $\boldsymbol a$ and the coupling $\hbar$; throughout,
$q=\omega_p$.  The top row $x_{n,i}=a_i$ is therefore not part
of the sum.  Each $T_{r,u}$ changes only $x_{r,u}$ to $qx_{r,u}$.
External K\"ahler transport varies $\boldsymbol\zeta$ and
$\boldsymbol z$ with $\boldsymbol a$ and $\hbar$ fixed.
The internal shifts produce the orbit
\begin{equation}\label{cyc:finite-grid}
 \mathcal O=\{\boldsymbol x_{\boldsymbol j}
     =(q^{j_\alpha}x_\alpha^0)_{\alpha=1}^R:
        \boldsymbol j\in(\Z/p\Z)^R\}.
\end{equation}
The $p^R$ points are distinct.  Indeed, a product of shifts
fixes $\boldsymbol x^0$ precisely when each exponent is divisible
by $p$; thus the orbit is \emph{free}.
We assign a nonzero scalar
$\Phi_\alpha:\mathcal O\to\C^\times$ to every directed edge
$\boldsymbol x\to T_\alpha\boldsymbol x$ and define transport
along a path by the product of these scalars.  Flatness identifies
the products along the two paths across a coordinate square:
\begin{equation}\label{cyc:square-flatness}
 \Phi_\alpha(\boldsymbol x)\Phi_\beta(T_\alpha\boldsymbol x)
 =\Phi_\beta(\boldsymbol x)\Phi_\alpha(T_\beta\boldsymbol x).
\end{equation}
Consequently transport on the lifted grid is independent of the
order of the steps.  After $p$ steps in direction $\alpha$ we
return to the same vertex, but the weight is multiplied by
$\mathcal B_{p,\alpha}$.  Flatness around squares therefore leaves a separate
closure condition on each full coordinate loop.

For a function $f$ on $\mathcal O$ let its twisted difference be
\begin{equation}\label{cyc:twisted-difference}
 (\Delta_\alpha f)(\boldsymbol x)
   =\Phi_\alpha(\boldsymbol x)f(T_\alpha\boldsymbol x)
     -f(\boldsymbol x).
\end{equation}
Suppose the weights satisfy
$\mu(T_\alpha\boldsymbol x)=\Phi_\alpha(\boldsymbol x)\mu(\boldsymbol x)$,
so that $\mu(\boldsymbol x)\Delta_\alpha f(\boldsymbol x)$ is
the difference of adjacent weighted values of $f$.  The following
criterion determines when the weights are periodic and these
differences sum to zero.

\begin{theorem}\label{cyc:closure}
Let $\mathcal O$ be the grid \eqref{cyc:finite-grid}, with
nonzero edge factors $\Phi_\alpha:\mathcal O\to\C^\times$
satisfying \eqref{cyc:square-flatness}.  The following conditions
are equivalent:
\begin{enumerate}
 \item $\mathcal B_{p,\alpha}(\boldsymbol x^0)=1$ for every internal direction;
 \item there is a nonzero weighted evaluation functional
 \[
  \mathcal S[f]=\sum_{\boldsymbol x\in\mathcal O}
                      \mu(\boldsymbol x)f(\boldsymbol x)
 \]
 such that $\mathcal S[\Delta_\alpha f]=0$ for every $\alpha$
 and every function $f$ on $\mathcal O$.
\end{enumerate}
Such a functional is unique up to a scalar and has no zero
weights.  The same criterion is obtained by testing only rational
functions regular on the orbit.
\end{theorem}

\begin{proof}
Comparing the coefficient of $f(T_\alpha\boldsymbol x)$ in
$\mathcal S[\Delta_\alpha f]=0$, we obtain
\begin{equation}\label{cyc:weight-recursion}
 \mu(T_\alpha\boldsymbol x)=\Phi_\alpha(\boldsymbol x)\mu(\boldsymbol x).
\end{equation}
The orbit is connected and all edge factors are nonzero, so one
nonzero weight determines nonzero weights at every point.
Iteration of \eqref{cyc:weight-recursion} around a coordinate
loop forces $\mathcal B_{p,\alpha}=1$.
Conversely, flatness allows us to interchange adjacent steps,
and $\mathcal B_{p,\alpha}=1$ makes every full coordinate loop trivial.
Starting with $\mu(\boldsymbol x^0)=1$, the recursion therefore
defines weights on $\mathcal O$ independently of the chosen
path.  Reindexing the sum gives
$\mathcal S[\Delta_\alpha f]=0$, and the same recursion proves
uniqueness up to normalization.
Finally, the orbit consists of distinct points.  Polynomial
interpolation realizes arbitrary values on this finite set,
which proves the assertion about rational test functions.
\end{proof}

Geometrically, the Bethe equations make the holonomy of the
rank-one local system of internal transport trivial.  The
weighted orbit is then a closed cycle with coefficients in this
local system, and the theorem is a finite version of Stokes'
formula.  It gives a general formulation of the cancellation
mechanism in \cite{TV26} (Sections~3 and~5), using the Bethe
equations and discrete exactness already present there.

To write the weights explicitly, order the internal indices as
$\alpha=1,\ldots,R$ and set, on the lifted lattice,
\begin{equation}\label{cyc:internal-transport}
 \mu_{\boldsymbol j}(\boldsymbol z,\boldsymbol a)=
 \prod_{\alpha=1}^R\prod_{d=0}^{j_\alpha-1}
 \Phi_\alpha
   (T_1^{j_1}\cdots T_{\alpha-1}^{j_{\alpha-1}}T_\alpha^d
       \boldsymbol x^0;\boldsymbol z,\boldsymbol a),
 \qquad 0\leq j_\alpha<p.
\end{equation}
Here $\mu_{\boldsymbol0}=1$.  At a Bethe point the weights
are periodic in each $j_\alpha$.  Before imposing closure, the
sum in one direction has the boundary term, with
$\mathcal B_p=\prod_{d=0}^{p-1}T^d\Phi$,
\begin{equation}\label{cyc:boundary-defect}
 \sum_{j=0}^{p-1}\mu_j
       \bigl(\Phi(T^j\boldsymbol x^0)f(T^{j+1}\boldsymbol x^0)
                     -f(T^j\boldsymbol x^0)\bigr)
  =(\mathcal B_p-1)f(\boldsymbol x^0).
\end{equation}
For several directions we apply this calculation to each
coordinate line, keeping the transverse weights fixed.  The
opposite faces of the weighted cube cancel precisely when their
coordinate holonomies are one.

\subsection{Finite Sums and the Frobenius tRS Equations}

The closed weights define finite vertex sums with insertions.
We shall use their K\"ahler multipliers and the classical
tRS/Bethe correspondence to obtain the tRS equations precomposed
with the Frobenius action.

Let $\boldsymbol\zeta(\boldsymbol m)=q^{\boldsymbol m}
\boldsymbol\zeta^0$, with $\boldsymbol z(\boldsymbol m)$ given
by \eqref{intro:kahler-coordinates}.  We keep $\boldsymbol a$
and the initial internal point $\boldsymbol x^0$ fixed and choose
the character
\[
 G(\boldsymbol m)=\prod_i\rho_i(\boldsymbol x^0)^{m_i}.
\]
For a solution of $\mathcal B_{p,\alpha}=1$, the finite functional is
\begin{equation}\label{cyc:functional-definition}
 \mathcal S_{\boldsymbol m}[f]
 =G(\boldsymbol m)
  \sum_{\boldsymbol j\in\{0,\ldots,p-1\}^{R}}
 \mu_{\boldsymbol j}(\boldsymbol z(\boldsymbol m),\boldsymbol a)
 f(T^{\boldsymbol j}\boldsymbol x^0;\boldsymbol\zeta(\boldsymbol m),\boldsymbol a),
 \quad R=\frac{n(n-1)}2.
\end{equation}
Here $f$ is the descendant insertion.  The choice $f=1$ gives
$\mathcal S_{\boldsymbol m}[1]$, the finite cyclic analogue of
the vertex integral with unit insertion.  Its phase is represented
by the weights $\mu_{\boldsymbol j}$ together with $G$.
We assume that $f$ is rational and regular on the full internal
orbit for every point of the external $q$-orbit; in particular,
$p_i^pf=f$.  Flatness and reindexing of the sum give
\begin{align}\label{cyc:functional-rules}
 \mathcal S_{\boldsymbol m}[\Delta_\alpha f]&=0,\notag\\
 \mathcal S_{\boldsymbol m+e_i}[f]
   &=\mathcal S_{\boldsymbol m}[\rho_i p_i f],\qquad
 \mathcal S_{\boldsymbol m+pe_i}[f]
   =\lambda_i\mathcal S_{\boldsymbol m}[f].
\end{align}
These relations connect the internal graph of vertices
\(\boldsymbol x_{\boldsymbol j}\) and edge factors
\(\Phi_\alpha\) with the external K\"ahler lattice
\(\boldsymbol m\in\Z^n\).  The Bethe equations close the
internal coordinate loops.  A step in external direction \(i\)
acts on the entire functional by insertion of \(\rho_i\), and
the mixed relation in \eqref{cyc:flatness} makes this action
compatible with internal transport and exact differences.
Notice that even though
\(\boldsymbol\zeta(\boldsymbol m+pe_i)
=\boldsymbol\zeta(\boldsymbol m)\), the functional acquires the
multiplier \(\lambda_i\).  Internal closure therefore allows
nontrivial external Bloch holonomy.

\begin{theorem}\label{cyc:finite-frobenius}
Let $q=\omega_p$ be a primitive root of any positive order $p$.
Let $\boldsymbol x^0$ belong to a regular closed internal orbit
with $\mathcal B_{p,\alpha}=1$ for all $\alpha$, and suppose that the powered
parameters and Bethe point are in the generic locus of the
classical tRS/Bethe correspondence.  For
$\lambda_i=\rho_i(\boldsymbol x^0)^p$ we have, for
$1\leq r\leq n$,
\begin{equation}\label{cyc:powered-level-set}
 h_r(\boldsymbol\zeta^p,\boldsymbol\lambda;\hbar^p)
       =e_r(\boldsymbol a^p).
\end{equation}
If $f$ is a rational insertion regular on the full internal and
external orbits and $p_i^pf=f$, then
$V_f(\boldsymbol m)=\mathcal S_{\boldsymbol m}[f]$ obeys
\begin{align}\label{cyc:finite-frobenius-equations}
 p_i^pV_f&=\lambda_iV_f,\notag\\
 H_r^{[p]}V_f
   &=\mathcal A_{p,r}(H_1,\ldots,H_n)V_f
     =e_r(\boldsymbol a^p)V_f.
\end{align}
These identities require no ordinary tRS eigenvalue equation and
remain valid when $V_f=0$.  Whenever $V_f$ is nonzero, it is a
joint eigenfunction of the tRS family precomposed with the
Frobenius action.
\end{theorem}

\begin{proof}
Equation \eqref{cyc:powered-Bethe} identifies $\mathcal B_{p,\alpha}=1$ with
the Bethe equations for $\boldsymbol x^p$ at
$(\boldsymbol z^p,\boldsymbol a^p,\hbar^p)$.
We can therefore apply the classical correspondence of
\cite{KPSZ21} (Theorems~4.6 and~4.9).  With the normalization
following \eqref{cyc:quantum-momentum}, its momenta are given by
\eqref{cyc:flag-momenta}, which proves
\eqref{cyc:powered-level-set} without using a difference equation
for $V_f$.  In these homogeneous coordinates
$z_r^p=\zeta_r^p/\zeta_{r+1}^p$, since $q^p=1$.

The full-cycle relation in \eqref{cyc:functional-rules} is the
first equation of \eqref{cyc:finite-frobenius-equations}.
Since both the coefficients of $H_r^{[p]}$ and the multipliers
$\lambda_i$ are invariant under all external shifts, substitution
into the operator gives
\[
\begin{aligned}
 H_r^{[p]}V_f
 &=\hbar^{-pr(n-r)/2}
   \sum_{|I|=r}\prod_{\substack{i\in I\\j\notin I}}
       \frac{\hbar^p\zeta_i^p-\zeta_j^p}
            {\zeta_i^p-\zeta_j^p}
       \prod_{i\in I}\lambda_i\,V_f\\
 &=h_r(\boldsymbol\zeta^p,\boldsymbol\lambda;\hbar^p)V_f.
\end{aligned}
\]
Now \eqref{cyc:powered-level-set} and
Theorem~\ref{flag:frobenius} give the second equation.  The
calculation does not involve division by $V_f$.
\end{proof}

We have thus obtained the Frobenius equations from Bethe closure
for every regular insertion, including the unit class and the
descendants below.  This gives a finite-sum construction alongside
the eigenfunction implication of
Corollary~\ref{flag:eigenfunction-frobenius}.

Let us illustrate the cancellation for $T^*\Fl_3$, where the
internal variables are $x_{1,1},x_{2,1},x_{2,2}$.  For $p=2$ the
sum has eight terms.  On any internal coordinate line denote the
weights by $\mu_0,\mu_1$, the transport ratios by
$\Phi_0,\Phi_1$, and the values of a test function by $f_0,f_1$.
Then
\[
 \mu_0(\Phi_0f_1-f_0)+\mu_1(\Phi_1f_0-f_1)
 =(\mu_0\Phi_0-\mu_1)f_1
   +(\mu_1\Phi_1-\mu_0)f_0=0.
\]
The transport relation makes the first coefficient zero, and the
Bethe equation $\Phi_0\Phi_1=1$ makes the second zero.  There
are four such lines in each internal direction.  For this flag
the K\"ahler multipliers take the form
\[
 \lambda_1=\hbar^p x_{1,1}^p,\qquad
 \lambda_2=\left(\frac{x_{2,1}x_{2,2}}{x_{1,1}}\right)^p,\qquad
 \lambda_3=\hbar^{-p}
       \left(\frac{a_1a_2a_3}{x_{2,1}x_{2,2}}\right)^p.
\]

\subsection{Explicit Descendant Components}
\label{cyc:explicit-components}

We now express the unit-insertion vertex as a sum of explicit
descendant components.  The descendants form a partition of
unity, and their finite sums satisfy the Frobenius equations of
Theorem~\ref{cyc:finite-frobenius}.

In \cite{KZ21} (Theorem~5.11, equations~(5.25)--(5.26)),
stable-envelope insertions give the components of a geometric
qKZ solution.  We use the same interpretation of components as
vertices with insertions, taking interpolation descendants to
decompose the finite closed-cycle sum.

With $L_0=1$, introduce the successive determinant ratios
\[
 \beta_r=\frac{L_r}{L_{r-1}},\qquad 1\leq r<n.
\]
For pairwise distinct equivariant parameters $a_1,\ldots,a_n$
and $\sigma\in S_n$, define the rational descendant
\begin{equation}\label{cyc:interpolation-weights}
 w_\sigma(\boldsymbol x;\boldsymbol a)
 =\prod_{r=1}^{n-1}\prod_{s=r+1}^{n}
       \frac{\beta_r-a_{\sigma(s)}}
            {a_{\sigma(r)}-a_{\sigma(s)}}.
\end{equation}
By a \emph{descendant insertion} we mean a symmetric Laurent
polynomial in the tautological roots, with coefficients in the
localized equivariant coefficient field, inserted into the vertex
integral or finite sum.  Here $\beta_r=L_r/L_{r-1}$ represents
the line-bundle class
$\det\mathcal V_r\otimes(\det\mathcal V_{r-1})^{-1}$.
Consequently $w_\sigma$ is a Laurent polynomial in determinant
characters whose coefficients are rational functions of the
$a_i$.  Equivariant localization inverts the factors
$a_{\sigma(r)}-a_{\sigma(s)}$, which accounts for the term
\emph{rational descendant}.  The expression is symmetric in the
roots of each $\mathcal V_r$ separately and is regular on the
internal torus for distinct $a_i$.  We use the displayed
representative as the insertion.
The functional $\mathcal S_{\boldsymbol m}$ contains the
homogeneous character and its centering, so we set
\begin{equation}\label{cyc:explicit-component-definition}
 V(\boldsymbol m)=\mathcal S_{\boldsymbol m}[1],
 \qquad
 I_\sigma(\boldsymbol m)=\mathcal S_{\boldsymbol m}[w_\sigma].
\end{equation}
Since $w_\sigma$ is independent of the external coordinates and
regular on the internal torus for distinct $a_i$, we can apply
Theorem~\ref{cyc:finite-frobenius} to each $I_\sigma$ and to $V$.
In particular,
\[
 H_r^{[p]}I_\sigma=e_r(\boldsymbol a^p)I_\sigma,
 \qquad p_i^pI_\sigma=\lambda_iI_\sigma.
\]

\begin{proposition}
\label{cyc:explicit-component-sum}
The functions \eqref{cyc:interpolation-weights} obey the rational
identity
\begin{equation}\label{cyc:partition-of-unity}
 \sum_{\sigma\in S_n}w_\sigma=1.
\end{equation}
Their classes give the fixed-point interpolation basis of
localized $K_T(T^*\Fl_n)$.  As polynomials in the independent
variables $\beta_1,\ldots,\beta_{n-1}$, they also form a basis of
the space defined by $\deg_{\beta_r}\leq n-r$.
For any finite-cycle functional satisfying
\eqref{cyc:functional-rules}, we have
\begin{equation}\label{cyc:explicit-component-sum-formula}
 V=\sum_{\sigma\in S_n}I_\sigma,\qquad
 I_\sigma=\mathscr P_\sigma(\boldsymbol p)V,
\end{equation}
where the difference polynomial is given explicitly by
\begin{equation}\label{cyc:component-difference-polynomial}
 \mathscr P_\sigma(\boldsymbol p)
 =\prod_{r=1}^{n-1}\prod_{s=r+1}^{n}
   \frac{\hbar^{\,r-(n+1)/2}p_r-a_{\sigma(s)}}
        {a_{\sigma(r)}-a_{\sigma(s)}}.
\end{equation}
Thus $\sum_\sigma\mathscr P_\sigma(\boldsymbol p)=1$ is an
identity of difference operators.
\end{proposition}

\begin{proof}
For a finite set $A$ of distinct scalars, interpolation of the
constant polynomial gives
\[
 \sum_{a\in A}\prod_{c\in A\setminus\{a\}}
                  \frac{b-c}{a-c}=1.
\]
Let us fix $\sigma(1)$ in the sum
\eqref{cyc:partition-of-unity}.  Applying this identity
recursively to the remaining parameters and the variables
$\beta_2,\ldots,\beta_{n-1}$ shows that the inner sum is one.
Its application to $\beta_1$ then proves
\eqref{cyc:partition-of-unity}.  No relations among the
$\beta_r$ enter this calculation.

At the fixed flag $\tau$ we have $\beta_r=a_{\tau(r)}$.
For $\sigma\neq\tau$, let $r$ be the first index at which they
differ.  Then $a_{\tau(r)}$ belongs to
$a_{\sigma(r+1)},\ldots,a_{\sigma(n)}$, so one factor of
$w_\sigma|_\tau$ vanishes.  On the other hand,
$w_\tau|_\tau=1$.  Equivariant localization identifies these
classes with the coordinate basis of the fixed-point set.
The same evaluations prove linear independence of the $n!$
polynomials $w_\sigma$.  Since the space with the prescribed
degree bounds has dimension $\prod_{r=1}^{n-1}(n-r+1)=n!$,
they form a basis of that space as well.

Applying the linear functional to the partition of unity gives
the first equality in \eqref{cyc:explicit-component-sum-formula}.
Its external transport rule, together with
$\rho_r=\hbar^{(n+1)/2-r}\beta_r$, realizes insertion of
$\beta_r$ by the operator $\hbar^{r-(n+1)/2}p_r$.
Both the descendants and the $\rho_r$ are independent of the
external coordinates.  Repeating these insertions therefore
gives \eqref{cyc:component-difference-polynomial}, and substituting
into \eqref{cyc:partition-of-unity} proves the last assertion.
\end{proof}

For $n=2$ this gives the two descendants
$(x-a_2)/(a_1-a_2)$ and $(x-a_1)/(a_2-a_1)$.
The partition of unity applies equally to a vertex integral
whenever these insertions are defined.  In every rank the
components \(I_\sigma\) use one cycle, one phase, and one
normalization; only their insertions vary.  Thus \(\sigma\)
indexes the entries of a single column, rather than independently
chosen solutions.
On a regular closed orbit satisfying the assumptions of
Theorem~\ref{cyc:finite-frobenius}, this construction gives
\[
 V=\sum_{\sigma\in S_n}I_\sigma,
 \qquad H_r^{[p]}V=e_r(\boldsymbol a^p)V,
 \qquad p_i^pV=\lambda_iV.
\]
Each component $I_\sigma$ obeys the same powered equations.
For a joint ordinary eigenfunction $H_rV=e_r(\boldsymbol a)V$,
these equations instead follow directly from the operator
identity of Theorem~\ref{flag:frobenius}.  This is
Corollary~\ref{flag:eigenfunction-frobenius}, valid in every rank
and at every root order.


\begin{bibdiv}
\raggedright
\begin{biblist}
\bib{C94}{article}{
 author={Cherednik, Ivan},
 title={Integration of Quantum Many-Body Problems by Affine Knizhnik--Zamolodchikov Equations},
 journal={Adv. Math.},
 volume={106},
 pages={65--95},
 date={1994},
}
\bib{D21}{article}{
 author={Dinkins, Hunter},
 title={$3d$ Mirror Symmetry of the Cotangent Bundle of the Full Flag Variety},
 date={2021},
 eprint={2011.08603},
}
\bib{EV24a}{article}{
 author={Etingof, Pavel},
 author={Varchenko, Alexander},
 title={Periodic and Quasi-Motivic Pencils of Flat Connections},
 date={2024},
 eprint={2401.00636},
}
\bib{EV24b}{article}{
 author={Etingof, Pavel},
 author={Varchenko, Alexander},
 title={$p$-Curvature of Periodic Pencils of Flat Connections},
 date={2024},
 eprint={2401.05652},
}
\bib{K21}{article}{
 author={Koroteev, Peter},
 title={A-Type Quiver Varieties and ADHM Moduli Spaces},
 journal={Commun. Math. Phys.},
 volume={381},
 number={1},
 pages={175--207},
 date={2021},
 eprint={1805.00986},
}
\bib{K23}{article}{
 author={Koroteev, Peter},
 title={Quantum Geometry, Integrability, and Opers},
 date={2023},
 eprint={2312.17500},
}
\bib{KPSZ21}{article}{
 author={Koroteev, Peter},
 author={Pushkar, Petr P.},
 author={Smirnov, Andrey},
 author={Zeitlin, Anton M.},
 title={Quantum $K$-Theory of Quiver Varieties and Many-Body Systems},
 journal={Selecta Math. (N.S.)},
 volume={27},
 number={5},
 note={Article 87},
 date={2021},
 eprint={1705.10419},
}
\bib{KS26}{article}{
 author={Koroteev, Peter},
 author={Smirnov, Andrey},
 title={Quantum $K$-Theory of Quiver Varieties at Roots of Unity},
 journal={Int. Math. Res. Not. IMRN},
 date={2026},
 number={14},
 note={rnag153},
 eprint={2412.19383},
}
\bib{KZ21}{article}{
 author={Koroteev, Peter},
 author={Zeitlin, Anton M.},
 title={qKZ/tRS Duality via Quantum $K$-Theoretic Counts},
 journal={Math. Res. Lett.},
 volume={28},
 number={2},
 pages={435--470},
 date={2021},
 eprint={1802.04463},
}
\bib{M92}{article}{
 author={Matsuo, Atsushi},
 title={Integrable Connections Related to Zonal Spherical Functions},
 journal={Invent. Math.},
 volume={110},
 number={1},
 pages={95--122},
 date={1992},
}
\bib{R87}{article}{
 author={Ruijsenaars, S. N. M.},
 title={Complete Integrability of Relativistic Calogero--Moser Systems and Elliptic Function Identities},
 journal={Commun. Math. Phys.},
 volume={110},
 pages={191--213},
 date={1987},
}
\bib{RS86}{article}{
 author={Ruijsenaars, S. N. M.},
 author={Schneider, H.},
 title={A New Class of Integrable Systems and Its Relation to Solitons},
 journal={Ann. Physics},
 volume={170},
 number={2},
 pages={370--405},
 date={1986},
}
\bib{SV23}{article}{
 author={Smirnov, Andrey},
 author={Varchenko, Alexander},
 title={The $p$-Adic Approximations of Vertex Functions via $3D$-Mirror Symmetry},
 date={2023},
 eprint={2302.03092},
}
\bib{TV26}{article}{
 author={Tarasov, Vitaly},
 author={Varchenko, Alexander},
 title={Monodromy Eigenvectors for Difference Equations with Root-of-Unity Step},
 date={2026},
 eprint={2609.17742},
}
\bib{VV10}{article}{
 author={Varagnolo, Michela},
 author={Vasserot, Eric},
 title={Double Affine Hecke Algebras at Roots of Unity},
 date={2010},
 eprint={math/0603744},
}
\end{biblist}
\end{bibdiv}
\end{document}